\documentclass[a4paper,10pt]{article}

\usepackage[margin=3cm]{geometry} % page setup
\usepackage{amsmath, amsfonts, amssymb, amsthm}
\usepackage{graphicx}
\usepackage{caption} % om breedte caption in tabellen en figuren aan te passen
\usepackage[all]{xy} % See Tobias Oetiker, p.102
\newcommand{\Nats}{\mathbb{N}}
\newcommand{\Real}{\mathbb{R}}
\newcommand{\I}{\mathrm{i}}
\newcommand{\half}{{\frac{1}{2}}}

\newcommand{\FT}{\mathcal{F}}
\newcommand{\e}{\mathrm{e}}
\newcommand{\supp}{\operatorname{supp}}

\newcommand{\Op}{\operatorname{Op}}

\newcommand{\wfs}{\operatorname{WF}}
\newcommand{\cotbun}[1]{T^\ast\!{#1}\!\setminus\!0}   % cotangent bundle
\newcommand{\Conic}{\Real^n\!\setminus\!\{0\}}        % conic for subset
\newcommand{\tines}{\!\times\!}

\newcommand{\pdpd}[2]{\tfrac{\partial #1}{\partial #2}}
\newcommand{\pdpdn}[2]{\frac{\partial #1}{\partial #2}}

\newcommand{\bfx}{{\mathbf{x}}}
\newcommand{\bfy}{{\mathbf{y}}}
\newcommand{\bfv}{{\mathbf{v}}}
\newcommand{\bfz}{{\mathbf{z}}}
\newcommand{\bfxi}{{\boldsymbol{\xi}}}
\newcommand{\bfeta}{{\boldsymbol{\eta}}}
\newcommand{\bftheta}{{\boldsymbol{\theta}}}
\newcommand{\bfzeta}{{\boldsymbol{\zeta}}}
\newcommand{\bfp}{{\mathbf{p}}}

\newcommand{\bfn}{{\mathbf{n}}}
\newcommand{\bfns}{{\mathbf{n}_{\rm s}\hspace{-0.4mm}}}
\newcommand{\rma}{{\rm a}}
\newcommand{\rmb}{{\rm b}}
\newcommand{\rmc}{{\rm c}}
\newcommand{\rmh}{{\rm h}}
\newcommand{\rmr}{{\rm r}}
\newcommand{\rms}{{\rm s}}

\newcommand{\mat}[2]{\begin{pmatrix} #1 \\ #2 \end{pmatrix}}

\newcommand{\phiT}{{\varphi_{\rm \hspace{-0.2mm}_T\hspace{-0.4mm}}}}

\newtheorem{thm}{Theorem}
\newtheorem{lem}{Lemma}

\title{Linearized inverse scattering based on seismic Reverse Time Migration%
\footnote{First presented at the Conference on Applied Inverse
  Problems, June 28, 2007}} \author{Tim J.P.M. Op 't Root, Christiaan
  C. Stolk, Maarten V. de Hoop}

\date{December 20, 2010}

\begin{document}
%\setvruler[12.5pt][1][1][2][0][20pt][20pt][0pt][\textheight]
\maketitle

\begin{abstract}
In this paper we study the linearized inverse problem associated with
imaging of reflection seismic data. We introduce an inverse scattering
transform derived from reverse-time migration (RTM). In the process,
the explicit evaluation of the so-called normal operator is avoided,
while other differential and pseudodifferential operator factors are
introduced. We prove that, under certain conditions, the transform
yields a partial inverse, and support this with numerical simulations.
In addition, we explain the recently discussed `low-frequency
artifacts' in RTM, which are naturally removed by the new method.
\end{abstract}

\section{Introduction}

In reflection seismology one places point sources and point receivers
on the earth's surface. A source generates acoustic waves in the
subsurface, which are reflected where the medium properties vary
discontinuously. In seismic imaging, one aims to reconstruct the
properties of the subsurface from the reflected waves that are
observed at the surface \cite{Claerbout1985Book,Biondi2006,Symes2009}.
There are various approaches to seismic imaging, each based on a
different mathematical model for seismic reflection data with
underlying assumptions. In general, seismic scattering and inverse
scattering have been formulated in the form of a linearized inverse
problem for the medium coefficient in the acoustic wave equation. The
linearization is around a smoothly varying background, called the
velocity model, which is a priori also unknown. However, in the
inverse scattering setting considered here, we assume the background
model to be known. The linearization defines a single scattering
operator mapping the model contrast (with respect to the background)
to the data, that consists of the restriction to the acquistion set of
the scattered field. The adjoint of this map defines the process of
imaging in general. The composition of the imaging operator with the
single scattering operator yields the so-called normal operator, the
properties of which play a central role in developing an inverse
scattering procedure.
 
There are different types of seismic imaging methods. One can
distinguish methods associated with the evolution of waves and data in
time from those associated with the evolution in depth (or another
principal spatial direction). The first category contains approaches
known under the collective names of Kirchhoff migration
\cite{BleisteinCohenStockwell2001} or generalized Radon transform
inversion, and reverse-time migration (RTM)
\cite{SchultzSherwood1980,whitmore:1983,mcmechan:1983,baysalKS:1983,sunM:2001};
the second category comprises the downward continuation approach
\cite{Clay,Claerbout1985Book,Biondi1996,Pop96,Jin2002} possibly
applied in curvilinear coordinates. The analysis pertaining to inverse
scattering in the second category can be found in Stolk and De Hoop
\cite{stolkdH:2005,stolkdH:2006}. The subject of the present paper is
an analysis of RTM-based inverse scattering in the first category,
with a view to studying the reconstruction of singularities in the
contrast. As was done in the analysis of Kirchhoff methods
\cite{Beylkin85,Rakesh1988,KroodeSmitVerdel98,stolkdH:2002}, we make
use of techniques and concepts from microlocal analysis, and Fourier
integral operators (FIOs); see, e.g., \cite{Duistermaat1996} for
background information on these concepts. As through an appropriate
formulation of the wave field continuation approach, we arrive at a
representation of RTM in terms of a FIO associated with a canonical
graph. Over the past few years, there has been a revived interest in
reverse time migration (RTM), partly because their application has
become computationally feasible. RTM is attractive as an imaging
procedure because it avoids approximations derived from asymptotic
expansions or one-way wave propagation.
 
The study of the above mentioned normal operator takes into account
the available source-receiver acquisition geometry. To avoid the
generation of artifacts, one has to invoke the Bolker condition
\cite{Guillemin1985}, essentially ensuring that the normal operator is
a pseudodifferential operator. (In reflection seismology this
condition is sometimes referred to as traveltime injectivity condition
\cite{NolanSymes1997}.) RTM is based on a common source geometry, in
which case the Bolker condition requires the absence of ``source
caustics'', that is, caustics are not allowed to occur between the
source and the image points under consideration \cite{NolanSymes1997}.
We shall refer to the assumption of absence of source caustics as the
source wave multipath exclusion (SME). Additionally, we require that
there are no rays connecting the source with a receiver position,
which we refer to as the direct source wave exclusion (DSE), and we
exclude grazing rays that originate in the subsurface. These
conditions can be satisfied by removing the corresponding part of the
wavefield using pseudodifferential cutoffs.

In this paper we revisit the original reverse-time imaging
procedure. We do this, also, in the context of the integral
formulation of Schneider \cite{schneider:1978} and the inverse
scattering integral equation of Bojarski \cite{bojarski:1982}.  An RTM
migration algorithm constists of three main parts: The modeling of
source wave propagation in forward time, the modeling of receiver or
reflected wave propagation in reverse time, and the applicaton of the
so called imaging condition \cite{Claerbout1985Book,Biondi2006}. The
imaging condition is a map that takes as input the source wave field
and the backpropagated receiver wave field, and maps these to an
image. The imaging condition is based on Claerbout's
\cite{Claerbout1971} imaging principle: Reflectors exist in those
points in the subsurface where the source and receiver wave fields
both have a large contribution at coincident times.

Various imaging conditions have been developed over the past 25
years. The excitation time imaging condition identifies the time that
the source field passes an image point, for example, using its maximum
amplitude, and evaluates the receiver field at that time. The image
can be normalized by dividing by the source amplitude. Alternatively,
the image can be computed in the temporal frequency domain by dividing
the receiver field by the source field and integrating over frequency,
the ratio imaging condition. To avoid division by small values of the source field,
regularization techniques have been applied. An alternative is the
crosscorrelation imaging condition, in which the product of the fields
is integrated over time. Later other variants have been proposed, see
e.g.\ \cite{Chang1986,Chattopadhyay2008,KiyashchenkoEtAl2007}. The
authors of \cite{KiyashchenkoEtAl2007} use the spatial derivatives of
the fields, similarly to what we find in this work.

We introduce a parametrix for the linearized scattering problem on
which RTM is based. The explicit evaluation of the normal operator is
avoided, at the cost of introducing other pseudodifferential operator
factors in the procedure, which is, thus, different from Least-Squares
migration-based approaches \cite{plessixM:2004}. The method involves a
new variant of the ratio imaging condition that involves time
derivatives of the fields and their spatial gradients. The ratio
imaging condition, albeit a new variant, is hence finally provided
with a mathematical proof. The result is summarized in
Theorem~\ref{thm_R}. As an intermediate result, we also obtain a new
variant of the so called excitation time imaging condition in
Theorem~\ref{thm_tildeR}. Moreover, we also address the relation with
RTM ``artifacts''
\cite{yoonMS:2004,mulderP:2004,fletcherFKA:2005,Xie2006,guittonKB:2007},
as well as certain simplifications that occur when dual sensor
streamer data are available.

The seismic waves are governed by the acoustic wave equation with
constant density on the spatial domain $\Real^n$ with $n=1,2,3$, given
by
\begin{equation} \label{eq:wave_equation}
   \big[ c(\bfx)^{-2} \partial_t^2 - \Delta \big] u(\bfx,t) 
    = f(\bfx,t).
\end{equation}
Although the subsurface is represented by the half space $\Real^{n-1}
\tines [0,\infty)$, we carry out our analysis in the full space,
  $\Real^n$. The acquisition domain is a subset of the surface
  $\Real^{n-1}\tines\{0\}$. The slowly varying velocity is a given
  smooth function $c(\bfx)$. The existence, uniqueness and regularity
  of solutions can be found in \cite{Hormander94}. We use the Fourier
  transform: \hbox{$\FT
    u(\bfxi,\omega)=\iint\e^{-\I(\bfxi\cdot\bfx+\omega t)}u(\bfx,t) \,
    d\bfx dt$}, and sometimes write $\widehat{u}$ for $\FT u$.

The outline of the paper is as follows. In
section~\ref{sec_solution_ivp}, solutions of the wave equation are
discussed, starting from the WKB approximation with plane wave initial
values. The (forward) scattering problem is analyzed in
section~\ref{sec_forward_scattering}. We focus on the map from the
contrast (or ``reflectivity'') to what we refer to as the continued
scattered field, which is the result from a perfect backpropagation of
the scattered field from its Cauchy values at some time after the
scattering has taken place. We obtain an explicit expression which is
locally valid, and a global characterization as a Fourier integral
operator. In section~\ref{sec_up_down_PsDO} we study the revert
operator, which describes the backpropagation of the receiver
field. The relation with the continued scattered field is
established. The inversion, that is, parametrix construction, is
presented in section~\ref{sec_inverse_scattering}. We first carry out
a brief analysis of the case of a constant velocity. Then we introduce
a novel version of the excitation time imaging condition and show that it
yields an inversion. Following that, we present an imaging condition
expressed entirely in terms of the source and backpropagated receiver
fields, providing the RTM based linearized inversion. In
section~\ref{sec_numerics} we show some numerical tests. We end the
paper with a short discussion.

\section{Asymptotic solutions of the initial value problem}  
\label{sec_solution_ivp}

In this section, we study solutions of the wave equation with smooth
coefficients. We introduce explicit expressions for the solution
operator for wave propagation over small times. In subsection~\ref{subsec_plane_wave}
we construct an approximate solution of the
IVP of the homogeneous wave equation.
Using the WKB approximation we introduce phase and
amplitude functions, which are solved by the method of characteristics
in subsections \ref{subsec_phase_function} and
\ref{subsec_amplitude_function}. The asymptotic solution is finally
written as a FIO in subsection \ref{subsec_FIO_S}. Subsection
\ref{subsec_decoupling} presents the decoupling of the wave equation
and general solution operators. Subsection
\ref{subsec_scattering_model} deals with the source field problem of
RTM.

\subsection{WKB approximation with plane-wave initial values}
\label{subsec_plane_wave}

Instead of solving (\ref{eq:wave_equation}) directly, we solve for
$c^{-1} u$, and consider the equivalent wave equation,
\begin{equation} \label{eq:norm_wave_eq}
   [ \partial_t^2 - c \Delta c ] (c^{-1} u) = 0 .
\end{equation}
In the later analysis it will be advantageous that $c \Delta c$ is a
symmetric operator. We invoke the WKB ansatz,
\begin{equation} \label{eq:WKB_ansatz}
   c^{-1} u (\bfx,t) = a(\bfx,t) \e^{\I\lambda{\alpha}(\bfx,t)} .
\end{equation}
A straightforward calculation yields
\begin{equation} \label{eq:WKB_rhs_expansion}
\begin{split}
   \e^{-\I\lambda{\alpha}}
      [\partial_t^2 -c\Delta c]\,{a}\,\e^{\I\lambda{\alpha}} 
  = & - \lambda^2{a}\big[(\partial_t{\alpha})^2-c^2|\nabla{\alpha}|^2\big] 
\\
   & + \I\lambda\big[2(\partial_t{a})\partial_t{\alpha} 
     + {a}\partial_t^2{\alpha} - 2c\nabla(c{a})\cdot\nabla{\alpha} 
     - c^2{a}\Delta{\alpha}\big] \\
   & + \partial_t^2{a} - c\Delta(c{a}).
\end{split}
\end{equation}
An approximate solution of the form (\ref{eq:WKB_ansatz}) is obtained
by requiring first that the term $O(\lambda^2)$ vanishes, resulting in
an eikonal equation for $\alpha$, and secondly that the term
$O(\lambda)$ also vanishes, resulting in a transport equation for
$a$. We will give these equations momentarily, and comment below on
the vanishing of terms $O(\lambda^j)$ for $j \le 0$.

We solve (\ref{eq:norm_wave_eq}) with plane-wave initial values:
\begin{equation} \label{equ_IVPFIO_initcond}
   u(\bfx,0) = 0 ,\qquad\qquad
   c(\bfx)^{-1} \partial_t u(\bfx,0) = \e^{\I \bfx \cdot \bfxi} .
\end{equation}
The role of $\lambda$ is here played by $| \bfxi |$. The WKB type
solution of the initial value problem will contain two terms,
i.e., the ansatz becomes
\begin{equation}  \label{equ_IVPFIO_WKB}
   c^{-1} u (\bfx,t)
   = {a}(\bfx,t;\bfxi)\, \e^{\I{\alpha}(\bfx,t;\bfxi)} 
     + {b}(\bfx,t;\bfxi)\, \e^{\I{\beta}(\bfx,t;\bfxi)}.
\end{equation}
The reason is that there is a sign choice in the equation for $\alpha$, leading to the eikonal equations
\begin{equation}  \label{equ_IVPFIO_eikonal}
   \partial_t{\alpha}+c|\nabla{\alpha}| = 0  \quad\quad\mathrm{and}\quad\quad
   \partial_t{\beta}-c|\nabla{\beta}| = 0.
\end{equation}
Here, ${\alpha}$ covers the negative frequencies and ${\beta}$ the
positive ones. The transport equations can be concisely written in
terms of ${a}^2$ and ${b}^2$. They are
\begin{equation}  \label{equ_IVPFIO_transport}
   \partial_t({a}^2\partial_t{\alpha}) - \nabla\cdot({a}^2c^2\nabla{\alpha}) = 0  \quad\quad\mathrm{and}\quad\quad
   \partial_t({b}^2\partial_t{\beta}) - \nabla\cdot({b}^2c^2\nabla{\beta}) = 0.
\end{equation}

The WKB ansatz (\ref{equ_IVPFIO_WKB}) can be inserted into the initial
conditions (\ref{equ_IVPFIO_initcond}). This straightforwardly yields
initial conditions for $\alpha,\beta$:
\begin{equation}  \label{equ_IVPFIO_inipha}
   {\alpha}(\bfx,0;\bfxi) = {\beta}(\bfx,0;\bfxi) = \bfxi\cdot\bfx.
\end{equation}
The initial conditions for $a$, $b$ can be given in the form of a
matrix equation,
\begin{equation*}\label{equ_IVPFIO_iniamp}
   \begin{pmatrix} 1 & 1 \\ -\I c(\bfx)|\bfxi| & \I c(\bfx)|\bfxi|
        \end{pmatrix} 
   \begin{pmatrix} {a}(\bfx,0;\bfxi) \\ {b}(\bfx,0;\bfxi) \end{pmatrix} 
   = \begin{pmatrix} 0 \\ 1 \end{pmatrix}
\end{equation*}
The two terms in (\ref{equ_IVPFIO_WKB}) are not independent. The
initial value problem for ${\alpha}$ can be transformed into the
initial value problem for ${\beta}$ by replacing $\bfxi$ with $-\bfxi$
and setting ${\beta}(\bfx,t;\bfxi)=-{\alpha}(\bfx,t;-\bfxi)$. Further
analysis shows that ${b}(\bfx,t;\bfxi)\,
\e^{\I{\beta}(\bfx,t;\bfxi)}$ in (\ref{equ_IVPFIO_WKB}) is in
fact the complex conjugate of ${a}(\bfx,t;-\bfxi)\,
\e^{\I{\alpha}(\bfx,t;-\bfxi)}$.

\subsection{The phase function on characteristics}
\label{subsec_phase_function}

The method of characteristics \cite[section 3.2]{Evans98} will be used
to solve the eikonal and transport equations, as usual. We first solve
the initial value problem for ${\alpha}(\bfy,t;\bfxi)$,
cf.\ (\ref{equ_IVPFIO_eikonal}) and (\ref{equ_IVPFIO_inipha}). The
same procedure can be applied to ${\beta}$.

The characteristic equations are formulated in terms of $(\bfy,t)$,
$(\bfp,\omega)$ associated with $(\nabla \alpha, \partial_t \alpha)$,
and a variable $q$ associated with $\alpha$. The eikonal equation is
hence given by
\begin{equation}
   F(\bfy,t, \nabla \alpha, \partial_t \alpha, \alpha) = 0 ,\qquad
   F(\bfy,t,\bfp,\omega,q) = \omega + c(\bfy) | \bfp | .
\end{equation}
The characteristic equations are then
\begin{equation}  \label{equ_IVPFIO_MOC_charode}
\begin{split}
   \frac{d}{ds} \begin{pmatrix} \bfy \\ t \end{pmatrix}
   = \begin{pmatrix} \frac{c \bfp}{| \bfp |} \\ 1 \end{pmatrix} ,
\qquad\qquad
   \frac{d}{ds} \begin{pmatrix} \bfp \\ \omega \end{pmatrix}
   = \begin{pmatrix} -(\nabla c) | \bfp | \\ 0 \end{pmatrix} ,
\qquad\qquad
   \frac{dq}{ds} = 0 .
\end{split}
\end{equation}
The only non-trivial equations are those for $\bfy$ and $\bfp$.
By $(\bfy(\bfx,t;\bfxi),
\bfp(\bfx,t;\bfxi))$ we denote a solution with $(\bfy(0),\bfp(0)) = 
(\bfx,\bfxi)$.

When $\alpha$ is a solution to (\ref{equ_IVPFIO_eikonal}),
(\ref{equ_IVPFIO_inipha}) on some open set $U \subset \Real^{n+1}$,
and $(\bfy(\cdot), t(\cdot))$ is a solution to the first equation
of (\ref{equ_IVPFIO_MOC_charode}), where 
$(\bfp(\cdot),\omega(\cdot),q(\cdot)) = 
(\nabla_\bfy \alpha(\bfy(\cdot),t(\cdot)),
\partial_t \alpha(\bfy(\cdot),t(\cdot)),
\alpha(\bfy(\cdot),t(\cdot)))$, then
$(\bfp(\cdot),\omega(\cdot),q(\cdot))$ solve the other equations
of (\ref{equ_IVPFIO_MOC_charode}), and in particular
$\alpha(\bfy(\bfx,s;\bfxi),s; \bfxi) = \alpha(\bfy(\bfx,0; \bfxi),0;
\bfxi)$. Differentiating this identity, and using the identity 
$(\partial \alpha/ \partial \bfy) \cdot (\partial \bfy/\partial \bfxi) = 0$, 
which is a consequence of the linearization of  
(\ref{equ_IVPFIO_MOC_charode}), it follows that
\begin{equation} \label{equ_x_dalpha_dxi}
  \text{if $\bfy = \bfy(\bfx,t;\bfxi)$ then }
  \partial_\bfxi \alpha(\bfy,t;\bfxi) = \bfx .
\end{equation}

To verify the local existence of solutions of
(\ref{equ_IVPFIO_eikonal}), (\ref{equ_IVPFIO_inipha}), one must derive
the initial conditions for (\ref{equ_IVPFIO_MOC_charode}) from
(\ref{equ_IVPFIO_eikonal}) and (\ref{equ_IVPFIO_inipha}) for each
point $\bfy$, and verify that these initial conditions are
noncharacteristic, i.e.\ $\partial_\omega F \neq 0$. The latter is
trivially the case. It follows therefore from \cite{Evans98} that
solutions exists up to some finite time locally, when
$\partial_\bfy\bfx$ becomes singular.

To examine the $\bfxi$-dependence of the constructed solution
${\alpha}$, we note that the initial conditions for
(\ref{equ_IVPFIO_MOC_charode}) depend in a smooth fashion on
$\bfxi$. Consequently, so does $\alpha$. Furthermore, a short
calculation shows that the function ${\alpha}(\bfy,t;\bfxi)$ is
positive homogeneous with respect to $\bfxi$ of degree one.

\subsection{The amplitude function} \label{subsec_amplitude_function}

In this subsection, we solve for the amplitude in terms of a Jacobian
of the flow of the rays. The result in equations
(\ref{eq:ampl_cons_eq}) and (\ref{equ_IVPFIO_MOC_amplitude}) is a
manifestation of the energy conservation property. The first step is
to carefully write equation (\ref{equ_IVPFIO_transport}) into the form
\begin{equation} \label{eq:ampl_eq1}
  0 = \left( \partial_t - \frac{c^2 \nabla \alpha}{\partial_t \alpha} 
                \cdot \nabla 
- \left(\nabla \cdot \frac{c^2 \nabla \alpha}{\partial_t \alpha} \right)\right) a^2
  =
  (\partial_t + \bfv \cdot \nabla + (\nabla \cdot \bfv)) a^2 .
\end{equation}
where we define $\bfv = - \frac{c^2 \nabla \alpha}{\partial_t
  \alpha}$. We used that $(\partial_t + \bfv \cdot \nabla) \partial_t
\alpha = 0$, i.e.\ the frequency is constant on a ray.
The field $\bfv$ is associated with the rays, which satisfy
\begin{equation}
  \frac{d\bfy}{dt}(t; \bfx) = \bfv(\bfy(t;\bfx),t) .
\end{equation}
We have $\frac{d}{dt} \frac{\partial \bfy}{\partial \bfx} =
\frac{\partial \bfv}{\partial \bfy} \frac{\partial \bfy}{\partial
  \bfx}$.  The derivative $\frac{d}{dt} \left| \frac{\partial
  \bfy}{\partial \bfx} \right|$ is hence related to $\nabla \cdot
\bfv$ as
\begin{equation}
  \frac{d}{dt} \det \left( \frac{\partial \bfy}{\partial \bfx}\right) 
  = \det \left(\frac{\partial \bfy}{\partial \bfx} \right) 
    \operatorname{tr} \left( \left(\frac{\partial \bfy}{\partial \bfx}
        \right)^{-1} 
        \frac{\partial}{\partial t} \frac{\partial \bfy}{\partial \bfx} \right)
  = (\nabla \cdot \bfv)
    \det \left(\frac{\partial \bfy}{\partial \bfx}\right) .
\end{equation}
This implies that
\begin{equation} \label{eq:ampl_cons_eq}
   \det(\partial_\bfx\bfy) \,  a^2
   \quad\text{ is constant along the ray.}
\end{equation}
Indeed, (\ref{eq:ampl_cons_eq}) is easily established by computing the
derivative $\frac{d}{dt} \left[ \det(\partial_\bfx\bfy(t;\bfx))
  a(\bfy(t;\bfx),t)^2 \right]$ and using (\ref{eq:ampl_eq1}).  From
(\ref{eq:ampl_cons_eq}) it follows that $a(\bfy(t;\bfx),t;\cdot) =
\sqrt{\det(\partial_\bfx\bfy(t;\bfx)^{-1})} \,
a(\bfx,0;\cdot)$. Inserting the $\bfxi$-dependence back into the
notation, and using that the map $\bfx \mapsto \bfy(\bfx,t;\bfxi)$ is
invertible results in
\begin{equation}  \label{equ_IVPFIO_MOC_amplitude}
  {a}(\bfy,t;\bfxi) 
  = \frac{\I}{2c(\bfx(\bfy,t;\bfxi))|\bfxi|} 
           \sqrt{\det(\partial_\bfy\bfx(\bfy,t;\bfxi))} .
\end{equation}

\subsection{Solution operator as a FIO}
\label{subsec_FIO_S}

In this subsection we consider more general initial values than
(\ref{equ_IVPFIO_initcond}) by considering linear combinations of the
terms in (\ref{equ_IVPFIO_WKB}). This results in an approximate
solution operator in the form of a Fourier integral operator (FIO)
\cite{Duistermaat1996,Treves80v1,Treves80v2,GrigisSjostrand1994} and
we will review some of its properties. Our solutions so far involve
only the highest order WKB terms and are limited to some small but
finite time.

We consider the original wave equation (\ref{eq:wave_equation}) with $f=0$
and the initial conditions
\begin{equation}
  u(\bfx,0) = 0 ,\qquad\qquad
  \partial_t u(\bfx,0) = h_2(\bfx) .
\end{equation}
Following (\ref{equ_IVPFIO_WKB}), its WKB solution for time $t\in I$,
which we will denote for the moment by $S_{12}(t) h_2(\bfy)$ is given
by a sum of two terms $S_{12}(t) h_2(\bfy) = c(\bfy) ( S_{\rm a2}(t)
h_2(\bfy) + S_{\rm b2}(t) h_2(\bfy))$, with
\begin{equation}  \label{equ_IVPFIO_solopr}
   S_{\rm a2}(t) h_2(\bfy) = \frac{1}{(2\pi)^n}\iint \e^{\I{\alpha}(\bfy,t;\bfxi)-\I\bfxi\cdot\bfx} \,  
                              \frac{{a}(\bfy,t;\bfxi)}{c(\bfx)}\, h_2(\bfx) \, d\bfx \, d\bfxi .
\end{equation}
Here the subscript ``${\rm a}$'' refers to the negative frequencies,
i.e.\ phase and amplitude functions $\alpha$ and~$a$. Then $S_{\rm b2}$ is
defined similarly, using $\beta$ and $b$, and refers to positive frequencies.
We recall that the symmetry relations of
subsection~\ref{subsec_plane_wave} imply that $S_{\rm b2}(t) h_2 =
\overline{S_{\rm a2}(t) h_2}$. The construction is such that $t$ can
be negative.

To argue that $S_{\rm a2}$ is a FIO, we will take a closer look at its
phase function, i.e.,
\begin{equation}  \label{equ_IVPFIO_phai}
   \varphi(\bfy,t,\bfx,\bfxi) = {\alpha}(\bfy,t;\bfxi) -\bfxi\cdot\bfx,
\end{equation}
and observe that it is positive homogeneous with respect to $\bfxi$ of
degree one, as it should. The stationary point set is given by
\begin{equation}  \label{equ_IVPFIO_statset}
    \Gamma_t = \big\{ (\bfy,\bfx,\bfxi)\in Y\tines X\tines\Conic \,\big|\, \bfx=\partial_{\bfxi}{\alpha}(\bfy,t;\bfxi) \big\}. %{\alpha}(\bfy,t;\bfxi) -\bfxi\cdot\bfx
\end{equation}
For $\Gamma_t$ to be a closed smooth submanifold of $Y\tines
X\tines\Conic$, the matrix,
\begin{equation*}
   \begin{pmatrix} \partial_\bfy\partial_\bfxi\varphi \\ \partial_\bfx\partial_\bfxi\varphi \\ \partial_\bfxi\partial_\bfxi\varphi \end{pmatrix}
   =  \begin{pmatrix} \partial_\bfy\partial_\bfxi{\alpha} \\ -{\rm I_n} \\ \partial_\bfxi\partial_\bfxi{\alpha} \end{pmatrix} ,
\end{equation*}
needs to have maximal rank on $\Gamma_t$, which is obviously the case
\cite[chapter VI, (4.22)]{Treves80v2}. The stationary point set
$\Gamma_t$ is hence a $2n$-dimensional manifold with coordinates
$(\bfy,\bfxi)$.

The stationary point set can be understood in terms of the
bicharacterstics. Definition (\ref{equ_IVPFIO_statset}) allows us to
express $\bfx$ on $\Gamma_t$ as a function
$\bfx_\Gamma(\bfy,t,\bfxi)=\partial_{\bfxi}{\alpha}(\bfy,t;\bfxi)$. Equation
(\ref{equ_x_dalpha_dxi}) implies that $(\bfy,\bfx,\bfxi)\in\Gamma_t$
if and only if a bicharacteristic initiates at $(\bfx,\bfxi)$ and
passes through $(\bfy,\bfeta)$ at time $t$ where $\bfeta$ must be
given by $\bfeta=\partial_\bfy{\alpha}(\bfy,t;\bfxi)$. If
$(\bfy,\bfx,\bfxi)\in Y\tines X\tines \Conic$ and $t\in\Real$ are such
that $(\bfy,\bfx,\bfxi)\in\Gamma_t$ then one has
$\partial_t{\alpha}(\bfy,t;\bfxi) = -c(\bfx) |\bfxi|$, since the
frequency $\partial_t{\alpha}$ is constant on a ray.

The propagation of singularities of $S_{\rm a2}$ is described by its
canonical relation,
\begin{equation}  \label{equ_IVPFIO_canrel}
   \Pi_t = \big\{ ((\bfy,\bfeta),(\bfx,\bfxi)) \in \cotbun{Y}\times\cotbun{X} \,\big|\, \bfx=\bfx_\Gamma(\bfy,t,\bfxi),\, \bfeta=\partial_\bfy{\alpha}(\bfy,t;\bfxi) \big\}.
\end{equation}
Clearly, $\Pi_t$ is the image of $\Gamma_t$ under the map
$(\bfy,\bfx,\bfxi)\mapsto((\bfy,\bfeta),(\bfx,\bfxi))$. It follows
from the characteristic ODE that the map from $(\bfx,\bfxi)$ to
$(\bfy,\bfeta)$ is a bijection,
$\Phi_t:\cotbun{X}\rightarrow\cotbun{Y}$ say. The canonical relation
is hence the graph of an invertible function. Therefore, each pair
$(\bfy,\bfxi)$, $(\bfx,\bfxi)$ and $(\bfy,\bfeta)$ can act as
coordinates on $\Gamma_t$, and on $\Pi_t$. We observe that $\Phi_t$
depends smoothly on $t$.

The effect of the FIO $S_{\rm a2}$ working on a distribution $v$ can
be explained in terms of the wave front set. If $v\in\mathcal{E}'(X)$,
then the wave front set $\wfs(v)$ of $v$ is a closed conic subset that
describes the {\em locations} and {\em directions} of the
singularities of $v$. Operator $S_{\rm a2}$ affects a distribution $v$
by propagating its wave front set by composition with the canonical
relation\cite{Duistermaat1996,Hormander90,Treves80v1,Treves80v2}. From
the above description of $\Pi_t$ it follows that
\begin{equation}
   \wfs(S_{\rm a2}(t) v)\subset \Phi_t(\wfs(v)).
\end{equation}

The pair $(\bfx,-\partial_\bfx\varphi)$ are referred to as the
\textsl{ingoing variable} and \textsl{covariable}, and
$(\bfy,\partial_\bfy\varphi)$ as the \textsl{outgoing variable} and
\textsl{covariable}.
%Illustrating this through solution operator $S_{\rm a2}$, we get 
%$-\partial_\bfx\varphi=\bfxi$ and, trivially, $\partial_\bfy\varphi=\bfeta$. 
The idea behind the names is that $S_{\rm a2}$, by $\Phi_t$, carries over $(\bfx,\bfxi)$ of the ingoing wave front set into $(\bfy,\bfeta)$ of the outgoing wave front set \cite[p.\ 334]{Treves80v2}.

%We worked out the solution operator for a finite interval of time. 
%Actually, the bicharacteristic flow can be defined for $t\in\Real$. 
%We will denote this flow by $\Phi_t$, which then forms a one-parameter 
%group of operators. With %this flow global solution operators can be 
%defined {\bf EXPAND THIS LAST PARAGRAPH !!! ... ???}\\

So far the highest order WKB approximation was used. The notion of
symbol classes for $a$, $b$ is needed to properly include lower order
terms. By replacing $a$ by an asymptotic sum $a(\bfx,t;\bfxi) =
\sum_{j=0}^\infty a_{m-j}(\bfx,t;\bfxi)$, with $a_k$ homogeneous of
order $k$ in $\bfxi$ for $|\bfxi| > 1$, the error in
(\ref{equ_IVPFIO_WKB}) can be made to decay as $|\bfxi|^{-N}$ for any
$N$. In other words, it becomes $C^\infty$ and the approximate
solution operator becomes a parametrix. Moreover, the \textsl{exact}
solution operator can be written in the form of $c (S_{\rm a2} +
S_{\rm b2})$ by the addition to $a$ and $b$ of certain symbols in
$S^{-\infty}$, which in particular decay faster than any power
$|\bfxi|^{-N}$ (unsurprisingly, the latter additions cannot be
computed with ray theory).

Solution operators for longer times have been constructed using more
general phase functions. For us those explicit expressions are of no
interest, but we note that the FIO property, with canonical relation
characterized by $\Phi_t$, remains valid, as can be seen by applying
the calculus of FIO's \cite[theorem 2.4.1]{Duistermaat1996} to the
product of several short time solution operators.

\subsection{Solution operators and decoupling}
\label{subsec_decoupling}

In subsection \ref{subsec_plane_wave} we assumed that the functions
${a}\e^{\I{\alpha}}$ and ${b}\e^{\I{\beta}}$ propagate independently
as solutions of the wave equation. In fact, this is the result of a
rather general procedure to decouple the wave equation
\cite{Taylor81}. Because the results of the decoupling will be used
explicitly in section~\ref{sec_up_down_PsDO} we give a short review of
it here; we will examine its relation to the solution operator $S_{\rm
  a2}$.

We write the wave field as the vector $(u_1(\bfx,t),u_2(\bfx,t))^T=(u,\partial_t u)^T$. The homogeneous wave equation can now be written as the following system, $1^{st}$-order with respect to time.
\begin{equation}    \label{equ_IVPFIO_system}
    \partial_t \begin{pmatrix} u_1 \\ u_2 \end{pmatrix} 
  = \begin{pmatrix} 0 & I \\ c(\bfx)^2\Delta & 0 \end{pmatrix}
    \begin{pmatrix} u_1 \\ u_2 \end{pmatrix} .
\end{equation}
The solution can be given as a matrix operator that maps the Cauchy data at $t=0$, say $(u_{0,1}(\bfx),u_{0,2}(\bfx))^T$, to the field vector at $t$.
\begin{equation}
   \mat{u_1}{u_2} = S(t) \mat{u_{0,1}}{u_{0,2}}   \quad\mathrm{with}\quad   S(t) = \mat{S_{11}(t) & S_{12}(t)}{S_{21}(t) & S_{22}(t)}.
\end{equation}
Naturally it satisfies the group property $S(t)S(s)=S(t+s)$. It is invertible by time reversal.

To decouple the system, we define several pseudodifferential operators. Let operator $B$ be a symmetric approximation of $\sqrt{-c(\bfx)\Delta c(\bfx)}$ with its approximate inverse $B^{-1}$ such that $B^2+c\Delta c$, $B^{-1}B-I$, and $BB^{-1}-I$ are regularizing operators, i.e.\ pseudodifferential operators of order $-\infty$. Although the square root does not necessarily have to be symmetric, being symmetric has the advantage that it yields a unitary solution operator, as we will see. Neglecting regularity conditions, we use symmetry and self-adjointness interchangeably. The principal symbols of $B$ and $B^{-1}$ are $c(\bfx)|\bfxi|$ and $\frac{1}{c(\bfx)|\bfxi|}$ respectively. The existence of such operators is a well known result in pseudodifferential operator theory, see e.g.\ \cite{OptRoot2010}. We now have the ingredients to define two matrix pseudodifferential operators $\Lambda$ and $V$ by
\begin{equation}  \label{equ_IVPFIO_VLambda}
    V = c(\bfx) \mat{1 & 1}{-\I B & \I B}  \quad\mathrm{and}\quad  \Lambda = \mat{1 & \I B^{-1}}{1 & -\I B^{-1}} \tfrac{1}{2c(\bfx)},
\end{equation}
which are each others inverses modulo regularizing operators. We finally define the following two fields $(u_{\rm a}(\bfx,t),u_{\rm b}(\bfx,t))^T = \Lambda (u_1,u_2)^T$. Note that the Cauchy data can be represented by a time evaluation of $(u_{\rm a},u_{\rm b})^T$. We will use the phrase `Cauchy data' in this way also. Omitting the regularizing error operators, the system (\ref{equ_IVPFIO_system}) transforms into a decoupled system for $(u_{\rm a},u_{\rm b})^T$ of which the first equation, together with its initial value, is
\begin{equation} \label{equ_IVPFIO_decoupled}
    \partial_t u_{\rm a} = -\I B u_{\rm a}  \quad\mathrm{and}\quad  u_{\rm a}(\bfx,0) = u_{\rm 0,a}(\bfx).
\end{equation}
By removing the minus sign it becomes the equation for $u_{\rm b}$. Let $S_{\rm a}$ and $S_{\rm b}$ be solution operators of the IVPs, i.e.\ $u_{\rm a}(\bfx,t)=S_{\rm a}(t)u_{\rm 0,a}(\bfx)$ and similar for~$S_{\rm b}$. Therefore, modulo regularizing operators
\begin{equation} \label{equ_IVFIO_relation_12ab}
   S(t) = V \mat{S_{\rm a}(t) & 0}{0 & S_{\rm b}(t)} \Lambda,
\end{equation}
which means that the original IVP (\ref{equ_IVPFIO_system}) and the decoupled system (\ref{equ_IVPFIO_decoupled}) have identical solutions disregarding a smooth error. Because $B$ is self-adjoint operators $S_{\rm a}$ and $S_{\rm b}$ are unitary, which follows from Stone's Theorem \cite{Conway1990}. It can be shown that $S_{\rm a}(t)$ and $S_{\rm b}(t)$ with $t\in\Real$ are FIOs \cite{Taylor1981}.

We turn to the relation of this matrix formalism and
$S_{12}=c(S_{\rma 2}+S_{\rmb 2})$, from which we derive a local
expression of $\mathrm{p.s.}(S_{\rm a})$, the principal symbol of
$S_{\rm a}$. The amplitude of $S_{\rm a2}$ is a homogeneous symbol,
which implies that it coincides with its principal symbol, and from
its definition~(\ref{equ_IVPFIO_solopr}) can thereafter be concluded
that $S_{\rm a2}=\mathrm{p.s.}(S_{\rm a}\Lambda_{12})$. The principal
symbol of a composition is the product of the principal symbols of its
factors \cite{Duistermaat1996,Treves80v2}, and hence
\mbox{$\frac{{a}(\bfy,t;\bfxi)}{c(\bfx)} = \mathrm{p.s.} (S_{\rm a})
  \frac{\I}{2c(\bfx)^2|\bfxi|}$}. Using the solution of the transport
equation (\ref{equ_IVPFIO_MOC_amplitude}), one concludes that
\begin{equation}    \label{equ_IVPFIO_psSa}
    \mathrm{p.s.} (S_{\rm a})(\bfy,\bfx,\bfxi) = \sqrt{\det(\partial_\bfy\bfx(\bfy,t;\bfxi))} .
\end{equation}
The principal symbol of $S_{\rm b}$ follows from $S_{\rm b}=\overline{S_{\rm a}}$.

\subsection{The source field}   \label{subsec_scattering_model}

In this subsection we discuss the source problem. The unperturbed
velocity is a smooth function $c(\bfx)$. The source wave is the
fundamental solution of a delta function located at $(\bfx_{\rm s},0)$
in space-time.
\begin{equation}  \label{equ_SCA_sourceproblem}
\begin{split}
   [c(\bfx)^{-2} \partial_t^2-\Delta] g(\bfx,t) &= \delta(\bfx-\bfx_{\rm s})\delta(t) \\
                                             g(\bfx,0) &= 0, \quad
                                             \partial_t g(\bfx,0) = 0.
\end{split}
\end{equation}

An important assumption is that 
\begin{equation} \label{eq:assume_SME}
  \text{the source wave does not exhibit multipathing (SME).}
\end{equation}
The fundamental solution can therefore be approximated by an asymptotic expansion with a single phase function. This can in principle be found by an application of section \ref{subsec_FIO_S} and using a change of phase function \cite[section 2.3]{Duistermaat1996}. One can show that, if $|\bfx-\bfx_{\rm s}|>\varepsilon$ for an $\varepsilon>0$ and $t$ bounded, the fundamental solution can be written as the Fourier integral \cite{Beylkin85}
\begin{equation}  \label{equ_SCA_full_asymptotic_source}
   g(\bfx,\bfx_{\rm s},t)
    = \frac{1}{2\pi}\int A(\bfx,\bfx_{\rm s},\omega) \, \e^{\I\omega(t-T(\bfx,\bfx_{\rm s}))} \, d\omega \,,
\end{equation}
with $A(\bfx,\bfx_{\rm s},\omega)\in S^\frac{n-3}{2}$ and $A(\bfx,\bfx_{\rm s},\omega) = \sum_{k=0}^\infty A_k(\bfx,\bfx_{\rm s},\omega)$. Each term is homogeneous, i.e.\ one has \mbox{$A_k(\bfx,\bfx_{\rm s},\lambda\omega) = \lambda^{\frac{n-3}{2}-k} A_k(\bfx,\bfx_{\rm s},\omega)$} for $\lambda>1$ and $|\omega|>1$. This holds for $n=1,2,3$. The sum means that for each $N\in\Nats$ there exists a $C_N>0$ such that
\begin{equation}  \label{equ_SCA_asymptotic_sum}
   \big|A(\bfx,\bfx_{\rm s},\omega)-{\sum}_{k=0}^{N-1} A_k(\bfx,\bfx_{\rm s},\omega)\big| \le C_N \, (1+|\omega|)^{\frac{n-3}{2}-N}.
\end{equation}
The source is real, implying that $\overline{A_k(\bfx,\bfx_{\rm s},\omega)}=A_k(\bfx,\bfx_{\rm s},-\omega)$ for all $k$. In (\ref{equ_SCA_full_asymptotic_source}) one can also view the separate contributions of positive and negative frequencies.

In part of the further analysis we will use the highest order term of the source field. There exist an amplitude $A_{\rm s}(\bfx)$ and a cutoff $\sigma(\omega)$, both real and such that \hbox{$A_0(\bfx,\bfx_{\rm s},\omega) = A_{\rm s}(\bfx)\,\sigma(\omega)\,(\I\omega)^{\frac{n-3}{2}}$} on the support of $\sigma$. Function $\sigma$ is smooth and has value 1 except for a neighborhood of the origin where it is 0. We also abbreviate \hbox{$T_{\rm s}(\bfx)=T(\bfx,\bfx_{\rm s})$}. The principal term of the expansion can now be written as
\begin{equation}  \label{equ_SCA_hot_source}
   g(\bfx,t) = A_{\rm s}(\bfx)\,\partial_t^{\frac{n-3}{2}} \delta(t-T_{\rm s}(\bfx)).
\end{equation}
Functions $A_{\rm s}(\bfx)$ and $T_{\rm s}(\bfx)$ will be referred to as the source wave amplitude and traveltime respectively. Operator $\partial_t^{\frac{n-3}{2}}$ denotes the pseudodifferential operator with symbol $\omega\mapsto\sigma(\omega)(\I\omega)^{\frac{n-3}{2}}$. The approximation $g(\bfx,t)$ matches the exact solution in case $\nabla c=0$ in the limit of $\omega\rightarrow\infty$. In that case one would have $T_{\rm s}(\bfx)=\tfrac{|\bfx-\bfx_\rms|}{c}$ and $A_{\rm s}(\bfx)=\tfrac{c}{2}, \sqrt{\tfrac{c}{8\pi|\bfx-\bfx_\rms|}}, \tfrac{1}{4\pi|\bfx-\bfx_\rms|}$ for respectively $n=1,2,3$~\cite{Beylkin85}. We define the \textsl{source wave direction} vector
\begin{equation}  \label{equ_SCA_swd_vector}
   \bfns(\bfx) = c(\bfx)\partial_\bfx T_{\rm s}(\bfx).
\end{equation}
This vector will, for example, be used to provide insight in the microlocal interpretation of the scattering event.

Source waves that arrive at the acquisition set are in the context of the inversion called direct waves. The negative frequency part of the wave front set of the source field is given by
\begin{equation}  \label{equ_directsourcerays}
   \Xi_\rms = \big\{ (\bfx,t,\bfxi,\omega)\in\cotbun{(X\tines\Real)}\ \big|\
                                       (\bfx,\bfxi)=\Phi_t(\bfx_\rms,\bfxi_\rms),\ \bfxi_\rms\in\Conic,\ \omega=-c(\bfx)|\bfxi| \big\}.
\end{equation}
It contains all bicharacteristics that go through $(\bfx_\rms,0)$ in spacetime. In the region where the Fourier integral (\ref{equ_SCA_full_asymptotic_source}) is valid, direct rays are also described by the equations $t=T_\rms(\bfx)$ and $\bfxi=|\bfxi|\,\bfns(\bfx)$. The restriction to time $t_\rmc$ is denoted by
\begin{equation}  \label{equ_directsourcerays_tc}
   \Xi_{\rms,t_\rmc} = \big\{ (\bfx,\bfxi)\in\cotbun{X}\ \big|\ (\bfx,t_\rmc,\bfxi,\omega) \in \Xi_\rms \big\}.
\end{equation}
This will be used to describe the direct waves in the Cauchy data of the continued scattered field.

\section{Forward scattering problem} \label{sec_forward_scattering}
We consider the scattering problem and formulate the \textsl{continued scattered wave field} as the result of the \textsl{scattering operator} acting on the \textsl{reflectivity}, i.e.\ the medium perturbation. We start with a description of the scattering model, essentially a linearization of the source problem. In section~\ref{subsec_continued_scat_wave} we derive an explicit expression for the mentioned operator. It will be used in section \ref{subsec_SCA_FIO} to define the global scattering operator, of which we show in theorem \ref{thm_Fa} that it is a FIO under the conditions of the DSE and the SME.

\subsection{Continued scattered wave field}  \label{subsec_continued_scat_wave}
Here, we introduce the \textsl{scattered wave field} and the \textsl{continued scattered wave field}. Loosely stated, the latter is the reverse time continuation of the former. We introduce the \textsl{scattering operator} that maps the medium perturbation to the continued scattered wave field. Theorem \ref{thm_localF} shows that a local representation of the operator can be written as an oscillatory integral.

The medium perturbation is modeled by the \textsl{reflectivity function} $r(\bfx)$. The non-smooth character of the perturbation gives rise to a scattered or reflected wave. We assume that 
\begin{equation} \label{eq:assume_supp_r}
  \supp(r) \subset D \text{ for a compact $D\subset\Real^{n-1}\tines[\epsilon,\infty)$ and some $\epsilon>0$.}
\end{equation}
The last component of $\bfx$ describes the depth. Because the source is at the surface, i.e.\ $x_{{\rm s},n}=0$, the reflectivity is zero in a neighborhood of the source. Following the Born approximation, the scattering problem is obtained by linearization of the source problem~(\ref{equ_SCA_sourceproblem}) with $(1+r(\bfx))c(\bfx)$ as the velocity. To find the linearization it is advantageous to first multiply (\ref{equ_SCA_sourceproblem}) with $c(\bfx)^2$. The result is
\begin{equation}  \label{equ_SCA_scatteringproblem}
\begin{split}
   [\partial_t^2-c(\bfx)^2\Delta] u(\bfx,t) &= r(\bfx) 2 A_{\rm s}(\bfx)\partial_t^{\frac{n+1}{2}}\delta(t-T_{\rm s}(\bfx)) \\
                                  u(\bfx,0) &= 0, \quad
                                  \partial_t u(\bfx,0) = 0.
\end{split}
\end{equation}
The \textsl{scattered wave field} $u(\bfx,t)$ is defined as the solution of the scattering problem (\ref{equ_SCA_scatteringproblem}).
We have used that the source wave field does not exhibit multipathing (SME) and can therefore be formulated as the asymptotic expansion~(\ref{equ_SCA_full_asymptotic_source}). In the forward modeling we will use the principal term to approximate the source, i.e.\ (\ref{equ_SCA_hot_source}). The subprincipal source terms do not contribute to the principal symbol of the scattering operator \cite{Rakesh1988}.

The \textsl{continued scattered wave field} $u_{\rm h}$ is defined as the solution of a final value problem of the homogeneous wave equation such that the Cauchy data at $t=T_1$ are identical with the Cauchy data of the scattered field $u$:
\begin{equation}  \label{equ_SCA_continuedscattered}
\begin{split}
   [\partial_t^2-c(\bfx)^2\Delta] u_{\rm h}(\bfx,t) &= 0 , \\
                                  u_{\rm h}(\bfx,T_1) &= u(\bfx,T_1), \quad
                                  \partial_t u_{\rm h}(\bfx,T_1) = \partial_t u(\bfx,T_1).
\end{split}
\end{equation}
The contributions to the scattered field entirely come to pass within the interval $[T_0,T_1]$, i.e.\ $T_0$ and $T_1$ are chosen such that $T_{\rm s}(\supp(r))\subset[T_0,T_1]$. For $t\geq T_1$ one has $u_{\rm h}(\bfx,t)=u(\bfx,t)$ but as $u_{\rm h}$ does and $u$ does not solve the homogeneous wave equation, they differ for $t<T_1$. We also use the decoupled wave fields $(u_{\rm h,a},u_{\rm h,b})^T = \Lambda (u_{\rm h},\partial_t u_{\rm h})^T$, with $\Lambda$ defined in (\ref{equ_IVPFIO_VLambda}).

The continued scattered wave field models the receiver wave field in an idealized experiment. Idealized here means that all scattered rays are present, even rays that do not intersect the acquisition set. It hence represents the scattered field by being its continuation in reverse time. The \textsl{reverse time continued wave field}, to be defined in section \ref{sec_up_down_PsDO}, models the receiver wave field.

%Formally, the value of $t_{\rm c}$ is irrelevant as a result of $u_{\rm h}$ being defined for all $t\in\Real$. It is used to pinpoint the Cauchy data of $u_{\rm h}$.

The \textsl{scattering operator} $F$ by definition maps $r$ to $(u_\rmh, \partial_t u_\rmh)^T$, and  we let $F_{\rm a}$ and $F_{\rm b}$ map the reflectivity $r$ to the decoupled components of the continued scattered wave field $u_{\rm h,a}$ and $u_{\rm h,b}$. To show that $F_{\rm a}$ is a FIO we derive an explicit formulation valid for a small time interval around a localized scattering event. Let $\{\rho_i\}_{i\in\mathcal{I}}$ be a finite smooth partition on $D$ such that $\sum_{i\in\mathcal{I}}\rho_i=1$ on $D$. Using $\rho_i$ as multiplication operator then
\begin{equation}  \label{equ_SCA_global}
   F_{\rm a}(t)={\sum}_{i\in\mathcal{I}} S_{\rm a}(t-t_{1i}) F_{\rm a}(t_{1i}) \rho_i,
\end{equation}
and $F_{\rm b}$ likewise. $S_{\rm a}$ is the solution operator (\ref{equ_IVFIO_relation_12ab}). The $i^{\rm th}$ local scattering event is delimited by $[t_{0i},t_{1i}]$, so $T_{\rm s}(\supp(\rho_i))\subset[t_{0i},t_{1i}]$. The partition is chosen fine enough such that $[t_{0i},t_{1i}]$ falls within an interval of definition of (\ref{equ_IVPFIO_solopr}), i.e.\ the local expression of solution operator $S_{\rma 2}$.

We write $\rho$ for an arbitrary member of $\{\rho_i\}_{i\in\mathcal{I}}$ and $[t_0,t_1]$ for its delimiting interval, and derive a local expression of the scattering operator evaluated at~$t_1$. We will prove the following
\begin{thm} \label{thm_localF}
   The local scattering operator $F_{\rm a}(t_1)\rho$ can be written as an oscillatory integral. It maps the reflectivity $r$ to the continued scattered wave field, that is, $u_{\rm h,a}(\bfy,t_1)=F_{\rm a}(t_1)\rho r(\bfy)$ and
   \begin{equation} \label{equ_SCA_Fa}
      u_{\rm h,a}(\bfy,t_1) = \frac{1}{(2\pi)^n}\iint \e^{\I\phiT(\bfy,t_1,\bfx,\bfxi)} \mathrm{A_F}(\bfy,t_1,\bfx,\bfxi) d\bfxi\, \rho r(\bfx) d\bfx,
   \end{equation}
   in which the phase and amplitude function are respectively defined as
   \begin{equation}    \label{equ_SCA_phaseamplitude}
   \begin{split}
      \phiT(\bfy,t_1,\bfx,\bfxi) &= {\alpha}(\bfy,t_1-T_{\rm s}(\bfx);\bfxi) - \bfxi\cdot\bfx, \\
      \mathrm{A_F}(\bfy,t_1,\bfx,\bfxi) &= 
(\I\partial_t{\alpha}(\bfy,t_1-T_{\rm s}(\bfx);\bfxi))^{\frac{n+1}{2}}\frac{{a}(\bfy,t_1-T_{\rm s}(\bfx);\bfxi)}{c(\bfx)} 2 A_{\rm s}(\bfx).
   \end{split}
   \end{equation}
\end{thm}
Here (\ref{equ_SCA_Fa}) is only the contribution of $\rho r$. There is a similar statement for $u_{\rm h,b}$, which satisfies $u_{\rm h,b}(\bfy,t)=\overline{u_{\rm h,a}(\bfy,t)}$. 
 
\begin{proof}
To solve the scattering problem~(\ref{equ_SCA_scatteringproblem}) it will be transformed into a $\tau$-parameterized family of IVP's. Duhamel's principle states that the solution, i.e.\ the scattered wave field, is given by
\begin{equation}  \label{equ_SCA_Duhamelsolvarc}
   u(\bfx,t):= \int_0^t \tilde u(\bfx,t;\tau)\,d\tau,
\end{equation}
in which for each $\tau$ function $\tilde u(\bfx,t;\tau)$ is the solution the homogeneous wave equation with prescribed Cauchy data on $t=\tau$ \cite[\S 2.4.2]{Evans98}:
\begin{equation} \label{equ_SCA_homivpvarc}
\begin{split}
   [\partial_t^2-c(\bfx)^2\Delta] \tilde u(\bfx,t;\tau) &= 0  \quad\mathrm{with}\quad  t\in\Real \\
             \tilde u(\bfx,\tau;\tau) &= 0 \\
   \partial_t\tilde u(\bfx,\tau;\tau) &=  r(\bfx) 2 A_{\rm s}(\bfx)\partial_t^{\frac{n+1}{2}}\delta(\tau-T_{\rm s}(\bfx)).
\end{split}
\end{equation}

The continued scattered wave field is the solution of the final value problem (\ref{equ_SCA_continuedscattered}). Using the observation that $r(\bfx) 2 A_{\rm s}(\bfx)\partial_t^{\frac{n+1}{2}}\delta(\tau-T_{\rm s}(\bfx))=0$ if $\tau\notin[T_0,T_1]$, it can be found by
\begin{equation}  \label{equ_SCA_h_wavefield}
   u_{\rm h}(\bfx,t) := \int_{T_0}^{T_1} \tilde u(\bfx,t;\tau)\, d\tau  \quad\mathrm{with}\quad  t\in\Real.
\end{equation}
Time integration is now over the fixed interval $[T_0,T_1]$, by which $u_{\rm h}$ solves the homogeneous wave equation. For $t\geq T_1$ the wave fields $u$ and $u_{\rm h}$ coincide. Therefore, this solves (\ref{equ_SCA_continuedscattered}).

To derive the local expression we solve the $\tau$-parameterized homogeneous IVP (\ref{equ_SCA_homivpvarc}) with $r$ replaced by $\rho r$ and evaluate the solution at~$t_1$. Let $(\tilde u_\rma,\tilde u_\rmb)^T=\Lambda(\tilde u,\partial_t\tilde u)^T$, then $\tilde u=c(\tilde u_\rma+\tilde u_\rmb)$. We apply solution operator $S_{\rma 2}$ with initial state at time $\tau$. This gives
\begin{equation}  \label{equ_SCA_h_wf}
   \tilde u_{\rm a}(\bfy,t_1;\tau) = S_{\rm a2}[\rho r(\bfx) 2 A_{\rm s}(\bfx)\partial_t^{\frac{n+1}{2}}\delta(\tau-T_{\rm s}(\bfx))](\bfy,t_1-\tau).
\end{equation}
Note that $S_{\rma 2}$ involves a \textsl{relative} time, i.e.\ the difference $t_1-\tau$, which is allowed because the medium velocity does not change in time. Then, time is as much as \textsl{absolute} when it agrees with the source time reference.

Consider $u_{\rm h,a}(\bfy,t_1)$, i.e.\ integral (\ref{equ_SCA_h_wavefield}) with $\tilde{u}$ replaced by $\tilde u_{\rm a}(\bfy,t_1;\tau)$ in (\ref{equ_SCA_h_wf}). We will eliminate $\tau$ by integration and write the field as an oscillatory integral. With the expression (\ref{equ_IVPFIO_solopr}) of $S_{\rm a2}$ and the application of $T_{\rm s}(\supp(\rho r))\subset[t_0,t_1]$ one derives the following integral
\begin{equation*}
   u_{\rm h,a}(\bfy,t_1) = \frac{1}{(2\pi)^n}\iiint
                        \e^{\I{\alpha}(\bfy,t_1-\tau;\bfxi)-\I\bfxi\cdot\bfx}\, \frac{{a}(\bfy,t_1-\tau;\bfxi)}{c(\bfx)} \rho r(\bfx)
                        2 A_{\rm s}(\bfx)\,\partial_t^{\frac{n+1}{2}}\delta(\tau-T_{\rm s}(\bfx)) \, d\bfx d\bfxi d\tau.
\end{equation*}
We recognize two convolutions, the integral over $\tau$ and operator $\partial_t^{\frac{n+1}{2}}$, the operator $\partial_t^{\frac{n+1}{2}}$ can be commuted to act on $\e^{\I{\alpha}-\I\bfxi\cdot\bfx}\tfrac{{a}}{c}$. Restricting to the highest order term, one writes $\partial_t^{\frac{n+1}{2}}[\e^{\I{\alpha}-\I\bfxi\cdot\bfx}\tfrac{{a}}{c}]=\nolinebreak[1](\I\partial_t{\alpha})^{\!\frac{n+1}{2}}\e^{\I{\alpha}-\I\bfxi\cdot\bfx}\tfrac{{a}}{c}\,$, which is an application of a general result of FIO theory \cite{Duistermaat1996,Treves80v2}. Cutoff $\sigma$ is omitted to shorten the expression. This yields
\begin{equation*}
   u_{\rm h,a}(\bfy,t_1) = \frac{1}{(2\pi)^n}\iiint
                        \left[\e^{\I{\alpha}-\I\bfxi\cdot\bfx}(\I\partial_t{\alpha})^{\!\frac{n+1}{2}}\frac{{a}}{c(\bfx)}\right]_{(\bfy,t_1-\tau;\bfxi)} \rho r(\bfx)
                        2 A_{\rm s}(\bfx)\,\delta(\tau-T_{\rm s}(\bfx)) \, d\bfx d\bfxi d\tau.
\end{equation*}
Notation $[\dots]_{\mathrm{arg}}$ means that ${\alpha}$, $\partial_t{\alpha}$ and ${a}$ within the square brackets are evaluated in given $\mathrm{argument}$. Explicit integration finally gives the oscillatory integral in (\ref{equ_SCA_Fa}), (\ref{equ_SCA_phaseamplitude}).
\end{proof}

\subsection{Scattering operator as FIO}   \label{subsec_SCA_FIO}
Here we establish that $F_{\rm a}(t_1)\rho$ is a FIO if the direct waves are excluded (DSE). We define the global scattering operator $\pi F$ and show that it is a FIO with an injective canonical relation, i.e.\ theorem \ref{thm_Fa}. 

Before we proceed with the theoretical aspects of the operator, we
will explain what it does. The stationary points of $F_{\rm
  a}(t_1)\rho$ are given by $\partial_{\bfxi}\phiT=0$,
i.e.\ $\partial_\bfxi{\alpha}(\bfy,t_1-T_{\rm
  s}(\bfx);\bfxi)-\bfx=0$. A stationary point $(\bfy,\bfx,\bfxi)$ has
the following interpretation. The source wave front excites the
reflectivity at $(\bfx,T_{\rm s}(\bfx))$ in space-time, causing a
scattering event. The event emits a scattered ray from $(\bfx,T_{\rm
  s}(\bfx))$ with initial covariable $\bfxi$, which arrives at
$(\bfy,t_1)$ with covariable
$\bfeta=\partial_\bfy\phiT(\bfy,t_1,\bfx,\bfxi)$. Operator $F_{\rm
  a}(t_1)\rho$ so describes the scattering event and the propagation
of the scattered wave over a small distance. The distance will be
extended by application of the solution operator, see
(\ref{equ_SCA_global}). Using the terminology introduced at the end of
subsection~\ref{subsec_FIO_S}, the ingoing variable and covariable are
$(\bfx,\bfzeta)$ with
$\bfzeta=-\partial_\bfx\phiT(\bfy,t_1,\bfx,\bfxi)$. The outgoing
variable and covariable are $(\bfy,\bfeta)$. This means that $F_{\rm
  a}(t_1)\rho$ carries over $(\bfx,\bfzeta)\in\wfs(r)$ into
$(\bfy,\bfeta)\in\wfs(u_{\rm h,a}(.,t_1))$.

We have
\begin{equation*}
   \bfzeta = -\partial_\bfx\phiT = \partial_t{\alpha}(\bfy,t_1-T_{\rm s}(\bfx);\bfxi) \,\partial_\bfx T_{\rm s}(\bfx) + \bfxi.
\end{equation*}
Using the source wave direction vector $\bfns(\bfx)=c(\bfx)\partial_\bfx T_{\rm s}(\bfx)$ and the identity $\partial_t{\alpha}=-c(\bfx)|\bfxi|$ for the frequency, this yields the relation between $\bfzeta$ and $\bfxi$,
\begin{equation} \label{equ_SCA_coivar}
   \bfzeta = \bfxi -|\bfxi|\,\bfns(\bfx) ,
\end{equation}
reflecting Snell's law. Figure \ref{fig_scattering} shows the microlocal picture of the scattering event and the scattered ray. Equation (\ref{equ_SCA_coivar}) also implies that $\bfzeta\cdot\bfns(\bfx)<0$ everywhere. This is a result of the geometry of the reflection event with one source. Note that (\ref{equ_SCA_coivar}) only holds for negative frequencies. For positive frequencies, i.e.\ considering $F_{\rm b}$, one gets $\bfzeta'=\bfxi+|\bfxi|\,\bfns(\bfx)$ instead. In that case $\bfzeta'\cdot\bfns(\bfx)>0$ everywhere.

\begin{figure}[ht!]
  \begin{center}
    \includegraphics[scale=1]{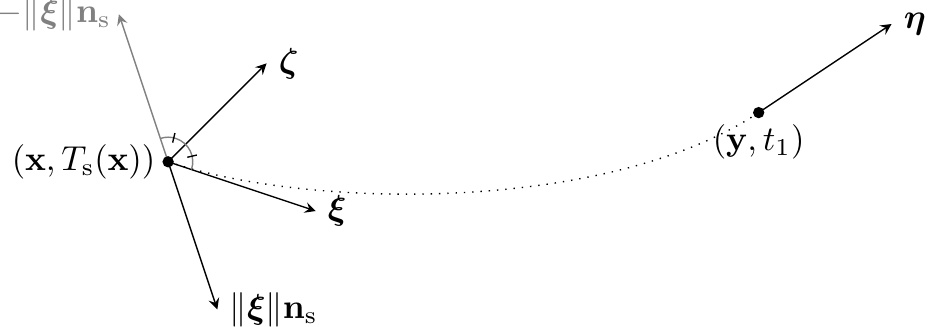}
  \end{center}
  \caption{Propagation of singularities at $(\bfx,T_{\rm s}(\bfx))$ 
    in space-time. See equation (\ref{equ_SCA_coivar}). The dotted line
    represents the ray. The endpoint of the ray at $(\bfy,t_1)$
    contributes to the scattered field. Here, $|\bfxi|\,\bfns$ and
    $\bfxi$ can respectively be interpreted as the wave numbers of the
    initial and reflected waves, and $\bfzeta$ a normal vector that can be
    associated with a reflector at $\bfx$.}
  \label{fig_scattering}
\end{figure}

If $(\bfx,\bfxi)$ is associated with a source ray, i.e.\ $\bfxi=|\bfxi|\,\bfns(\bfx)$, then $\bfzeta=0$ by (\ref{equ_SCA_coivar}). In that case there is no reflection. We show that away from the source rays the scattering operator $F$ is a FIO with an injective canonical relation, which will be made more precise. The practical implication is that source wave arrivals are excluded from the data before the receiver wave field is calculated.

The \textsl{direct source wave exclusion} (DSE) is the removal of the source singularities contained in $\Xi_\rms$ from the wave front set of the continued scattered wave field. Mathematically it will be applied by \mbox{$t$-families} of pseudodifferential operators $\pi_\rma(t)$ and $\pi_\rmb(t)$ that act on the Cauchy data $(u_{\rmh,\rma}(\cdot,t),u_{\rmh,\rmb}(\cdot,t))^T$. The symbol of $\pi_\rma(t_\rmc)$ is, for some fixed $t_\rmc$, a smooth cutoff function on $\cotbun{Y}$, being 0 on a narrow conic neighborhood of $\Xi_{\rms,t_\rmc}$ (cf.\ (\ref{equ_directsourcerays_tc})) and 1 outside a slightly larger conic neighborhood. Furthermore, we assume that $\pi_\rma$ satisfies 
\begin{equation} \label{equ_pi_a_family_property}
  \pi_\rma(t)=S_\rma(t-t_\rmc)\pi_\rma(t_\rmc)S_\rma(t_\rmc-t) ,
\end{equation}
which implies that the field $\pi_\rma u_{\rmh,\rma}$ still satisfies a homogeneous wave equation. The symbol $\pi_\rmb$ satisfies $\pi_\rmb(t;\bfx,\bfxi)=\pi_\rma(t;\bfx,-\bfxi)$.

Since, in the absence of multipathing, rays define paths of shortest traveltime between two points, we have the following property. Let $\bfx,\tilde\bfx\in D$ be not identical, then
\begin{equation}  \label{equ_SCA_TIC}
\begin{split}
%   & \mathrm{Let\ } \bfx,\tilde\bfx\in D,\ \bfx\neq\tilde\bfx.\ \ \\
   & \mathrm{if\ } \bfxi,\tilde\bfxi\in\Conic \mathrm{\ and\ } t_{\rm i}>0
         \mathrm{\ such\ that\ } (\tilde\bfx,\tilde\bfxi) = \Phi_{t_{\rm i}}(\bfx,\bfxi) \\
   & \mathrm{then\ } |T_{\rm s}(\tilde\bfx)-T_{\rm s}(\bfx)|<t_{\rm i} \mathrm{\ \ or\ \ } (\bfx,\bfxi)\in\Xi_{\rms,T_\rms(\bfx)}.
\end{split}
\end{equation}
If $\bfx$ and $\tilde\bfx$ lay on the same source ray then $|T_{\rm s}(\tilde\bfx)-T_{\rm s}(\bfx)|=t_{\rm i}$.

\smallskip

The central result is the theorem that the composition $\pi_\rma F_\rma$ is a FIO of which the canonical relation is the graph of an injective function. Let $V_{\rms,t}\subset\cotbun{Y}$ be the zero set of $\pi_\rma(t)$, a conic neighborhood of $\Xi_{\rms,t}$. With $\pi_\rmb F_\rmb=\overline{\pi_\rma F_\rma}$ we present the following
\begin{thm} \label{thm_Fa}
   Operator $\pi_\rma F_\rma$ defined above, is a FIO. Its canonical relation is
   \begin{multline}  \label{equ_SCA_canrel_global}
      \Lambda = \big\{ ((\bfy,t,\bfeta,\omega),(\bfx,\bfzeta))\ \big|\
                        (\bfy,\bfeta)\in(\cotbun{Y})\setminus V_{\rms,t},\ t\in\Real,\ \omega=-c(\bfy)|\bfeta|,\  \\
                        (\bfx,\bfxi)=\Phi_{T_{\rm s}(\bfx)-t}(\bfy,\bfeta),\ \bfzeta = \bfxi -|\bfxi|\,\bfns(\bfx)\,,\ \bfx\in D \big\}.
   \end{multline}
   The projection of $\Lambda$ to its outgoing variables, i.e.\ $(\bfy,t,\bfeta,\omega)$, is injective.
\end{thm}
We will first show that the composition $\pi_\rma(t_1)F_\rma(t_1)\rho$ is a FIO. Composition $\pi_\rma F_\rma$ is subsequently defined as the sum of local contributions, like in (\ref{equ_SCA_global}), and will also be called the `scattering operator'. The canonical relation becomes the union of the local relations. A part of the proof is put in lemma~\ref{lem_aroundscatpoint}. The operator can alternatively be defined by means of the bicharacteristics of the wave equation. The papers \cite{Rakesh1988,NolanSymes1997} show how this can be done, although their scattering operator does not fully coincide with ours.

\begin{proof} %$\pi_\rma(t_1)F_{\rm a}(t_1)\rho$ 
Because $\pi_\rma(t)S_{\rm a}(t-t_{1i})=S_{\rm a}(t-t_{1i})\pi_\rma(t_{1i})$ the scattering operator can be written as
\begin{equation}  \label{equ_SCA_piFglobal}
   \pi_\rma(t) F_{\rm a}(t) = {\sum}_{i\in\mathcal{I}} S_{\rm a}(t-t_{1i}) \pi_\rma(t_{1i}) F_{\rm a}(t_{1i}) \rho_i.
\end{equation}
Again omitting subscript $i$ to denote an arbitrary member of $\mathcal{I}$ we will argue that the local scattering operator $\pi_\rma(t_1)F_{\rm a}(t_1)\rho$ is a FIO. Then $\pi_\rma F_\rma$ becomes a sum of compositions of FIOs.

The local scattering operator is the oscillatory integral (\ref{equ_SCA_Fa}) in which the amplitude $\mathrm{A_F}$~(\ref{equ_SCA_phaseamplitude}) is replaced by $\pi_\rma(t_1;\bfy,\partial_\bfy{\alpha})\,\mathrm{A_F}$. This follows from the application of pseudodifferential operator $\pi_\rma(t_1)$, its symbol denoted by $\pi_\rma(t_1;\cdot,\cdot)$, on the integral \cite{Duistermaat1996,Treves80v2}. To be able to omit the zero set of $\pi_\rma(t_1)$ from the analysis of the phase $\phiT$ we define the conic set
% The direct waves expressed in the variables of the phase function is
% \begin{equation}
%    \Xi_\rms^{(pf)} = \big\{ (\bfy,t,\bfx,\bfxi)\in Y\tines\Real\tines X\tines\Conic\ |\
%                                           (\bfy,\bfeta)=\Phi_{t-T_\rms(\bfx)}(\bfx,\bfxi),\ \bfxi=|\bfxi|\,\bfns(\bfx)  \big\}.
% \end{equation}
% \begin{equation*}
%    W_{\rms,t} = \big\{ (\bfy,\bfx,\bfxi)\in Y\tines X\tines\Conic\ \big|\
%                (\bfx,\bfxi)\in V_{\rms,T_\rms(\bfx)},\ (\bfy,\bfeta)=\Phi_{t-T_\rms(\bfx)}(\bfx,\bfxi),\ \bfx\in D \big\}.
% \end{equation*}
\begin{equation*}
   W_{\rms,t_1} = \big\{ (\bfy,\bfx,\bfxi)\in Y\tines X\tines\Conic\ \big|\
               (\bfy,\bfeta)\in V_{\rms,t_1},\ (\bfx,\bfxi)=\Phi_{T_\rms(\bfx)-t_1}(\bfy,\bfeta),\ \bfx\in D \big\}.
\end{equation*}
The stationary point set of the phase function, by definition $\partial_\bfxi\phiT=0$, is given by
\begin{equation}  \label{equ_SCA_statset}
   \Sigma_{t_1} = \big\{ (\bfy,\bfx,\bfxi)\in (Y\tines X\tines\Conic)\setminus W_{\rms,t_1}\, \big|\
                                    \bfx=\partial_\bfxi{\alpha}(\bfy,t_1-T_{\rm s}(\bfx);\bfxi),\ \bfx\in\supp(\rho) \big\}.
\end{equation}
We observe that $|\partial_\bfx\partial_\bfxi\phiT| =
|\partial_\bfxi\partial_\bfx\phiT| = |\partial_\bfxi\bfzeta|$ by
definition of $\bfzeta$ (\ref{equ_SCA_coivar}). Moreover
\begin{equation}  \label{equ_SCA_dzetadxi}
   |\partial_\bfxi\bfzeta| = |{\rm I_n} - \frac{\bfxi}{|\bfxi|}
       \otimes \bfns(\bfx)| = 1 - \frac{\bfxi}{|\bfxi|}\cdot\bfns(\bfx) .
\end{equation}
By the DSE, applied as the omission of $W_{\rms,t_1}$ in
(\ref{equ_SCA_statset}), the condition $\bfxi\parallel\bfns(\bfx)$ is
never met, from which follows that the Jacobian
$|\partial_\bfxi\bfzeta|$ is nonsingular. This implies that the
derivative $\partial_{(\bfy,\bfx,\bfxi)}\partial_{\bfxi}\phiT$ has
maximal rank, making $\Sigma_{t_1}$ a closed smooth $2n$-dimensional
submanifold. The canonical relation relates ingoing (co)variables
$(\bfx,\bfzeta)$ with outgoing (co)variables $(\bfy,\bfeta)$ and is
given by
\begin{equation}  \label{equ_SCA_canrel}
   \big\{ ((\bfy,\bfeta),(\bfx,\bfzeta))\ \big|\
                        (\bfy,\bfx,\bfxi)\in\Sigma_{t_1},\ \bfeta=\partial_\bfy\phiT,\ \bfzeta=-\partial_\bfx\phiT  \big\}.
\end{equation}
The relation is the graph of a diffeomorphism. We postpone the proof until after the construction of the global scattering operator $\pi_\rma(t) F_\rma(t)$ as the local and the global arguments are basically the same. Therefore the local scattering operator is a FIO with a bijective canonical relation. 

%By composition with the solution operator the definition of the local scattering operator extends in time.
The local operator will be composed with the solution operator. This gives a seamless extension because both operators are build on the same flow. It becomes $S_{\rm a}(t-t_1) \pi_\rma(t_1) F_{\rm a}(t_1) \rho$, which is a FIO. The canonical relation is determined by the composition of relations \cite{Duistermaat1996,Treves80v2}. The global scattering operator $\pi_\rma(t)F_\rma(t)$ is subsequently defined as the sum (\ref{equ_SCA_piFglobal}) of the extended local operators, of which the canonical relation $\Lambda_t$ is the union of the local relations (\ref{equ_SCA_canrel}).

We will argue that $\Lambda_t$ is the graph of an injection \mbox{$\Theta_t:(\cotbun{D})\setminus U_\rms\rightarrow(\cotbun{Y})\setminus V_{\rms,t}$} that is a diffeomorphism onto its image. We used the zero set of $\pi_\rma(t)$ expressed in the domain of $\Theta_t$:
\begin{equation}
   U_\rms = \big\{ (\bfx,\bfzeta)\in\cotbun{X}\ \big|\
                           (\bfx,\bfxi)\in V_{\rms,T_\rms(\bfx)},\ \bfzeta = \bfxi -|\bfxi|\,\bfns(\bfx)\,,\ \bfx\in D  \big\}.
\end{equation}
The injection implies that $\Lambda_t$, for a fixed $t$, can be
parameterized by $\bfy$ and $\bfeta$, so
\begin{multline}  \label{equ_SCA_canrel_t_global}
   \Lambda_t = \big\{ ((\bfy,\bfeta),(\bfx,\bfzeta))\ \big|\
          (\bfy,\bfeta)\in(\cotbun{Y})\setminus V_{\rms,t},
\\
   (\bfx,\bfxi) = \Phi_{T_{\rm s}(\bfx)-t}(\bfy,\bfeta),\
        \bfzeta = \bfxi -|\bfxi|\,\bfns(\bfx)\,,\ \bfx\in D \big\} .
\end{multline}
We now prove the existence and injectivity of $\Theta_t$. Without loss
of generality we assume that $t$ denotes a moment after the scattering
event.

Let $(\bfx,\bfzeta)\in(\cotbun{D})\setminus U_\rms$ be given. It can
be shown that the transformation $\bfxi\mapsto\bfzeta$ given in
(\ref{equ_SCA_coivar}) is injective on the complement of $U_\rms$ and
thus determines a unique $(\bfx,\bfxi)$. By ray tracing over $t-T_{\rm
  s}(\bfx)$, i.e.\ mapping by $\Phi_{t-T_{\rm s}(\bfx)}$, one finds
$(\bfy,\bfeta)$.

Let $(\bfy,\bfeta)\in(\cotbun{Y})\setminus V_{\rms,t}$ be given. This
uniquely determines a bicharacteristic. By ray tracing backwards,
i.e.\ by $\Phi_{T_{\rm s}(\bfx)-t}$ with $T_{\rm s}(\bfx)-t<0$, the
ray goes through $(\bfx,T_{\rm s}(\bfx))$ in space-time. If a second
point $(\tilde\bfx,T_{\rm s}(\tilde\bfx))$ is met,
property (\ref{equ_SCA_TIC}) (SME) implies
that the bicharacteristic coincides with one from the source. The
condition $(\bfy,\bfeta)\notin V_{\rms,t}$ (DSE) rules out this
possibility, leading to the conclusion that $\bfx$ is unique. The
covariable $\bfxi$ uniquely follows from the ray tracing, and is
mapped to $\bfzeta$ by (\ref{equ_SCA_coivar}). The transformation
$\Theta_t$ is therefore one-to-one.

To prove the smoothness we analyse the scattering event around a fixed point $(\bfx_0,\bfxi_0)$, of which \mbox{$\bfx_0\in\supp(\rho)$}, and define $\tau_0=T_{\rm s}(\bfx_0)$. Now $\Theta_t$ can be factorized as follows
\begin{equation*}  \label{equ_SCA_XXX}
   (\bfx,\bfzeta) \xrightarrow{(\ref{equ_SCA_coivar})} (\bfx,\bfxi) \xrightarrow{\Phi_{\tau_0-T_{\rm s}(\bfx)}}
   (\check\bfx,\check\bfxi) \xrightarrow{\Phi_{t-\tau_0}} (\bfy,\bfeta).
\end{equation*}
The Jacobian of $\Theta_t$ becomes the product of three Jacobians, namely
\newcommand{\bv}{\Big|}
\begin{equation}
	\bv\frac{\partial(\bfy,\bfeta)}{\partial(\bfx,\bfzeta)}\bv \ = \ \bv\frac{\partial(\bfy,\bfeta)}{\partial(\check\bfx,\check\bfxi)}\bv
                                                                   \ \bv\frac{\partial(\check\bfx,\check\bfxi)}{\partial(\bfx,\bfxi)}\bv
                                                                   \ \bv\frac{\partial(\bfx,\bfxi)}{\partial(\bfx,\bfzeta)}\bv.
\end{equation}
The leftmost factor in the right hand side is nonsingular because $\Phi_{t-\tau_0}$ is a diffeomorphism. The rightmost factor in the right hand side is nonsingular because the map $\bfxi\mapsto\bfzeta$ has a positive Jacobian (\ref{equ_SCA_dzetadxi}). The transformation $\Phi_{\tau_0-T_{\rm s}(\bfx)}$ is the least obvious one. We will show in lemma \ref{lem_aroundscatpoint} that it is a smooth bijection. Therefore $\Theta_t$ is a diffeomorphism onto its image.
%The local canonical relation \equ{SCA_canrel} is the graph of $\Theta_{t_1}$ if its domain is restricted to $\supp(\rho)$.

So far $t$ was held fixed to simplify the presentation. Time dependence is determined by the flow $\Phi_t$. This allows $t$ to be included in the canonical relation $\Lambda$ of the scattering operator $\pi_\rma F_{\rm a}$, which is a map to spacetime distributions. Parameterized by $\bfy$, $\bfeta$ and $t$, $\Lambda$ becomes (\ref{equ_SCA_canrel_global}). The injectivity follows from the parameterization.
\end{proof}

\begin{lem}
Let $\tau_0=T_{\rm s}(\bfx_0)$ and $s(\bfx)=\tau_0-T_{\rm
  s}(\bfx)$. If $J(\bfx,\bfxi)=\Phi_s(\bfx,\bfxi)$ then $J$ is a
smooth bijection that maps $(\bfx_0,\bfxi_0)$ onto itself. Its
Jacobian is
\begin{equation}
    \det\partial_{(\bfx,\bfxi)}J(\bfx_0,\bfxi_0)
          = 1 - \frac{\bfxi_0}{|\bfxi_0|}\cdot\bfns(\bfx_0),
\end{equation}
which is nonsingular by the DSE.
\label{lem_aroundscatpoint}
\end{lem}

\begin{proof}
For $\bfx$ in the neighborhood of $\bfx_0$ one has $s(\bfx)\in I$, so
$\Phi_s$ is defined. The smoothness of $J$ follows directly from the
smoothness of $\bfx\mapsto T_{\rm s}(\bfx)$ and $\Phi_s$ in its
arguments including $s$. The Jacobian results from the straight
forward calculation
\begin{equation*}
\begin{split}
    \partial_{(\bfx,\bfxi)}J(\bfx_0,\bfxi_0)
        &= \partial_{(\bfx,\bfxi)}\Phi_0(\bfx_0,\bfxi_0) + \partial_s\Phi_0(\bfx_0,\bfxi_0) \otimes \partial_{(\bfx,\bfxi)}s(\bfx_0)    \\ 
        &= \mat{{\rm I_n} & 0}{0 & {\rm I_n}} + \mat{c(\bfx_0)\tfrac{\bfxi_0}{|\bfxi_0|}}{-|\bfxi_0|\, \partial_\bfx c(\bfx_0)} \otimes
            \big(-\partial_\bfx T_{\rm s}(\bfx_0)\ \ 0\,\big)                                                                                       \\
        &= \mat{{\rm I_n} - \tfrac{\bfxi_0}{|\bfxi_0|}\otimes\bfns(\bfx_0) & 0}
            {|\bfxi_0|\, \partial_\bfx c(\bfx_0)\otimes\partial_\bfx T_{\rm s}(\bfx_0) & {\rm I_n}}.
\end{split}
\end{equation*}
Herein we substitute the right-hand side of the characteristic
ODE~(\ref{equ_IVPFIO_MOC_charode}) for $\partial_s\Phi_s$.
\end{proof}

\section{Reverse time continuation from the boundary}
\label{sec_up_down_PsDO}

The receiver wave field is modeled by the reverse time continued wave
$u_\rmr$. In this section, we show that $u_\rmr$ is the result of a
pseudodifferential operator of order zero acting on the continued
scattered wave $u_\rmh$. We refer to it as the \textsl{revert
  operator} $P$.

The processes that are modeled by $P$ are the propagation of the
scattered wave field from a certain time, say $t=t_\rmc$, to the
surface at $x_n=0$, the restriction of the wave field to the
acquisition domain, the data processing, and eventually the
continuation in reverse time. The revert operator suppresses the part
of the scattered wave field that cannot be recovered because the
contributing waves do not reach the acquisition domain. The data
processing comprises a spatial smooth cutoff on the acquisition
domain, the removal of direct source waves and the removal of waves
reaching the surface following grazing rays. The final reconstruction
represents a field related to bicharacteristics that intersect the
acquisition domain $M$ only once, and in the upgoing direction.

Let $u$ be the solution to the homogeneous wave equation. When we
apply the result of this section to develope the inverse scattering, we
set $u = u_\rmh$. Let $M$ be a bounded open subset of $\{ (\bfx,t) \in
\Real^{n+1} \, | \, x_n = 0 \}$ and let $T_M u$ denote the restriction
of $u$ to $M$. We denote $\bfx' = (x_1,\ldots,x_{n-1})$, so
$(\bfx',t)$ are coordinates on $M$. The field $u_\rmr$ is an
anticausal solution to
\begin{equation} \label{eq:backprop_PDE}
  \left[ c(\bfx)^{-2} \partial_t^2 - \Delta\right] u_\rmr(\bfx,t)
    = \delta(x_n)  F_M T_M u(\bfx',t) ,
\end{equation}
where $F_M$ is defined as follows.

The boundary operator $F_M$ consists of two types of factors. A
pseudodifferential operator accounts for the fact that the boundary
data for the backpropagation enters as a source and not as a boundary
condition. This operator is given by
\begin{equation} \label{eq:one_way_factor}
  -2 \I D_t c^{-1} \sqrt{1- c^2 D_t^{-2} D_{\bfx'}^2} ,\quad
  D_t = \I^{-1} \partial_t ,\ D_{\bfx'} =  \I^{-1} \partial_{\bfx'} .
\end{equation}
The singularity in the square root is avoided by the cutoff for grazing
rays, see below. The second type of factor is composed of three cutoffs:
\renewcommand{\labelenumi}{(\roman{enumi})}
\begin{enumerate}
\itemsep=0pt
   \item
   The multiplication by a cutoff function that smoothly goes to zero
   near the boundary of the acquisition domain. The distance over
   which it goes from 1 to 0 in practice depends on the wavelengths
   present in the data.
   \item
   The second cutoff is a pseudodifferential operator which removes
   waves that reach the surface along tangently incoming rays. Its
   symbol is zero around $(\bfx',t,\bfxi',\omega)$ such that
   \[
     c(\bfy',0) | (\bfeta',0) | = \pm \omega  ,
   \]
   and 1 some distance away from this set.  If, given the velocity and
   the support of $\delta c$, there are no tangent rays, this cutoff
   is not needed.
   \item
   The third cutoff suppresses direct rays. Since the velocity model
   is assumed to be known, these can be identified.
\end{enumerate}
We write $\Psi_M(\bfx',t,\bfxi',\omega)$ for the symbol of the
composition of these pseudodifferential cutoffs. The principal symbol
of $F_M$ is then
\begin{equation} \label{eq:prin_sym_F_M}
  -2 \I \omega c^{-1} \sqrt{1- c^2 \omega^{-2} \bfxi'^2} \,
  \Psi_M(\bfx',t,\bfxi',\omega) .
\end{equation}

The decoupling procedure presented above yields two fields $u_{\rm a}$
and $u_{\rm b}$, associated respectively with the negative and
positive frequencies in $u$. We will show that $u_{\rmr,\rma}$ and
$u_{\rmr,\rmb}$ depend locally on $u_\rma$ and $u_\rmb$ in the
following fashion,
\begin{equation} \label{eq:time_RT_main}
\begin{split}
  \chi u_{\rmr,\rm a}(\cdot, t) 
  = {}& \chi \big[ P_{\rm a}(t) u_{\rm a}(\cdot.t) 
    \, + R_1(t) u_{\rm b}(\cdot,t) \big]
\qquad \text{ and } \\
  \chi u_{\rmr,\rm b}(\cdot, t) = {}& \chi \big[ P_{\rm b}(t) u_{\rm b}(\cdot,t)
    + R_2(t) u_{\rm a}(\cdot, t) \big].
\end{split}
\end{equation}
Here, $P_{\rm a}(t)$ and $P_{\rm b}(t)$ are pseudodifferential
operators described below and $R_1(t)$ and $R_2(t)$ are regularizing
operators, and $\chi$ is a cutoff because the source in equation (\ref{eq:backprop_PDE})
causes waves in both sides of $x_n=0$. Note that the decoupling, which so far was mostly a
technical procedure, turns out to be essential to characterize the
reverse time continued field. The revert operator in matrix form will
be defined as the \mbox{$t$-family} of pseudodifferential operators
\begin{equation}  \label{equ_RTM_operatorP}
   P(t) = V \mat{P_\rma(t) & R_1(t)}{R_2(t) & P_\rmb(t)} \Lambda .
\end{equation}
Waves are assumed to hit the set $M$ coming from $x_n > 0$. We assume
$\supp(\chi)$ to be compact and contained in the set with $x_n > 0$,
and we invoke the following assumption
\begin{equation} \label{eq:assume_rays_one-to-one}
  \text{bicharacteristics through $M$ and $\supp(\chi)$ intersect $M$
    only once and with $dx_n/dt < 0$.}
\end{equation}

The operators $P_{\rm a}$ and $P_{\rm b}$ depend on $F_M$ and on the
bicharacteristic flow in space-time between the hyperplanes $t=0$ and
$x_n=0$.  Let $X_s$ denote the set $\Real^n \times \{ s \}
\subset \Real^n_{\bfx} \times \Real_t$. The bicharacteristic flow
provides a map
\[
  (\bfx,\bfxi) \mapsto (\bfy_{\rm a}'(\bfx,\bfxi,t),t,
    \bfeta_{\rm a}'(\bfx,\bfxi,t),-c(\bfx)|\bfxi|) ,
\]
from $T^* X_0$ to $T^* M$. The principal symbols of $P_{\rm a}$,
$P_{\rm b}$, which we will denote by $p_\rma$, $p_\rmb$, are then
given by the following transported versions of $\Psi_M$:
\begin{equation} \label{eq:prin_sym_RTC}
  \Psi_{X_0,\rm a}(\bfx,\bfxi) = \left\{  
    \begin{array}{ll} \Psi_M(\bfy_{\rm a}'(\bfx,\bfxi,t),t,
              \bfeta_{\rm a}'(\bfx,\bfxi,t), -c(\bfx)|\bfxi|)
  & \text{when $\exists t$ with $\bfy_{\rm a}(\bfx,\bfxi,t) \in M$},
\\
  0 & \text{otherwise} ,
    \end{array} \right.
\end{equation}
and $\Psi_{X_0,\rm b}$ is defined similarly using the $(\bfy_{\rm b},
\bfeta_{\rm b})$ flow. We can now state and prove the following

\begin{thm} \label{thm_P}
Let $u_{\rmr,\rm a}$, $u_{\rmr,\rm b}$ and $u_{\rm a}$, $u_{\rm b}$,
$\chi$ and $M$ be as just defined. Equation (\ref{eq:time_RT_main}) holds, in
which $P_{\rm a}$ and $P_{\rm b}$ are pseudodifferential operators in
$\Op S^0(\Real^n)$, the principal symbols of which are given by
\begin{equation}
  p_\rma(t;\bfx,\bfxi) = \Psi_{X_0,\rm a}(\bfx,\bfxi)
\qquad \qquad \text{and} \qquad \qquad
  p_\rmb(t;\bfx,\bfxi) = \Psi_{X_0,\rm b}(\bfx,\bfxi) ,
\end{equation}
respectively. The operators $P_\rma$, $P_\rmb$ satisfy property
(\ref{equ_pi_a_family_property}) as far as they are uniquely
determined considering the cutoff $\chi$ in (\ref{eq:time_RT_main}).
\end{thm}

The proof will be presented in the remainder of this section.  If we
take Cauchy values at $t=t_\rmc$, then for small $| t-t_{\rmc} |$,
$T_M u(\bfx',t)$ can be described by the local FIO representation of
the solution operator. This representation can also be used for the
description of the map from $T_M u$ to $u_{\rmr}(\cdot,t_\rmc)$.  The
result can then be proven by an explicit use of the method of
stationary phase. For longer times we apply a partition of unity in
time to $T_M u(\bfx',t)$, so that for each contribution the length of
the time interval is small enough to apply the local FIO
representation. Egorov's theorem will be used to reduce to the short
time case. Alternatively one could consider one-way wave theory as a
method of proof.

\begin{proof}
We prove (\ref{eq:time_RT_main}) for some given $t$. Without loss of
generality we may assume that $t=0$. The field $u$, by assumption, solves the
homogeneous wave equation and is determined (possibly modulo a smooth contribution)
by the Cauchy values $u_{0,\rm a}=u_{\rma}(\cdot,0)$ and $u_{0,\rm b}=u_\rmb(\cdot,0)$.
Consider the equation $\partial_t u_{\rm a} = -\I B u_{\rm a}$. In this proof we
write $S_{\rm a}(t,s)$ instead of $S_{\rm a}(t-s)$ for the operator
that maps initial values at time $s$ to the values of the solution at
time $t$. We write $S_{\rm a}(\cdot, s)$ for the operator that maps an
initial value at time $s$ to the solution as a function of $(\bfx,t)$,
$t > s$.
We will write $S_{\rm a}(t,\cdot)$ for the operator that gives the
anticausal solution to $(\partial_t + \I B u_{\rm a})u_{\rm a} =
f_{\rm a}$,
\[
  S_{\rm a}(t,\cdot) f_\rma
    = - \int_t^\infty S_{\rm a}(t,s) f_{\rm a}(\cdot,s) \, ds .
\]
Note that $S_\rma(t,\cdot)$ maps a function of $(\bfx,t)$ to a function of $\bfx$ and that $S_{\rm a}(t,\cdot) = - S(\cdot, t)^*$.
The restriction operator $T_M$ introduced above maps
$C^\infty(\Real^n \times \Real) \rightarrow C^\infty(M)$ and is given
by
\[
  T_M u(\bfx',t) = u(\bfx',0,t) ,\qquad (\bfx',t) \in M .
\]
The adjoint of this operator is given by the following. With auxiliary function $f$ it is:
\[
  T_M^* f(\bfx,t) = \delta(x_n) f(\bfx',t) .
\]
These operators are well defined on suitable sets of distributions.
We use the notation (cf.\ (\ref{eq:backprop_PDE}))
\[
  f_M(\bfx',t) = F_M T_M u(\bfx',t)
\]
and study the map $(u_{0,\rm a},u_{0,\rm b}) \mapsto f_M$.  It follows
from the results on decoupling that
\begin{equation} \label{eq:formula_f_RTC}
  f_M = F_M T_M
    ( c S_{\rm a} u_{0,\rm a} + c S_{\rm b} u_{0,\rm b} ) ,
\end{equation}
modulo a smooth error. Following this decoupling, we analyze the map
$u_{0,\rm a} \mapsto F_M T_M c S_{\rm a} u_{0,\rm a}$.

To begin with, there exists a pseudodifferential operator $\tilde{F}_M$
such that
\[
  F_M T_M u =  T_M \tilde{F}_M u
\]
modulo a smooth function. This holds for a distribution $u$ that satisfies
$| \bfxi | \le C |\omega|$ in $\wfs(u)$ for some $C$, like the solution of
the homogeneous wave equation. Naturally,
$\tilde{F}_M(\bfx,t,\bfxi,\omega) \ne F_M(\bfx',t,\bfxi',\omega)$,
because then the symbol property would not be satisfied around the
line $(\bfxi',\omega)=0$, $\xi_n \neq 0$, but in the neighborhood of
this line the symbol can be modified without affecting the
singularities since $| \bfxi | \le C |\omega|$ in $\wfs(u)$. Thus the
first term in (\ref{eq:formula_f_RTC}) is given by
\begin{equation} \label{eq:first_term_RTC}
  T_M \tilde{F}_M c S_{\rm a}(\cdot,0) 
\end{equation}
acting on $u_{0,\rm a}$, which is a product of Fourier integral
operators.  The operator $S_{\rm a}(\cdot,0)$ has canonical relation
\begin{equation} \label{eq:RTC_CR1}
  \big\{ ((\bfy_{\rm a}(\bfx,\bfxi,t), t, \bfeta_{\rm a}(\bfx,\bfxi,t),-c(\bfx) |\bfxi|), 
    (\bfx,\bfxi)) \big\}.
\end{equation}
The operator $\tilde{F}_M$ removes singularities propagating on rays
that are tangent or close to tangent to the plane $x_n=0$, and the
restriction operator to $x_n=0$ has canonical relation
\begin{equation} \label{eq:RTC_CR2}
  \big\{ ((\bfy',t,\bfeta',\omega), (\bfy',0,t,\bfeta',\eta_n,\omega)) \big\} .
\end{equation}
As tangent rays are removed, the composition of canonical relations
(\ref{eq:RTC_CR2}) and (\ref{eq:RTC_CR1}) is transversal. Therefore,
(\ref{eq:first_term_RTC}) is a Fourier integral operator. Moreover,
from assumption (\ref{eq:assume_rays_one-to-one}) it follows that the
canonical relation is the graph of an invertible map, given by
\[
  \big\{ ((\bfy_{\rm a}'(\bfx,\bfxi,t), t, \bfeta_{\rm a}'(\bfx,\bfxi,t),-c(\bfx) |\bfxi|),
     (\bfx,\bfxi)) \, \big|\ \text{$t$ s.t.\ $\bfy_{\rm a}(\bfx,\bfxi,t) \in M$} \big\} ,
\]
or more precisely a subset of this set, taking into account the essential
support of $F_M$.

Next, we consider the map $f_M \mapsto \chi u_{\rmr,\rma}$. We insert a pseudodifferential cutoff \mbox{$\Xi(\bfx',t,D_{\bfx'},D_t)$}.
It cuts out tangent rays and is defined such that $\Xi F_M = F_M$. Using the decoupling procedure of section \ref{subsec_decoupling},
the source $(f_\rma,f_\rmb)$ for the inhomogeneous wave equation 
are given by $(f_\rma,f_\rmb)^T = \Lambda (0,c f)^T$, hence 
$\chi u_{\rmr,\rma}$ satisfies
\[
  \chi u_{\rmr,\rma}(\cdot,0) = \chi S_{\rm a}(0,\cdot) ( \tfrac{\I}{2} B^{-1}  c)
     T_M^*  \Xi  f_M
\]
There exists an operator $\tilde{\Xi}$ such that
$     T_M^* \Xi f = \tilde{\Xi}     T_M^* f$
at least microlocally on the set $| \bfxi | \le C |\omega|$ for large $C$.
Then $\chi u_{\rmr,\rma}(\cdot,0)$ is given by the operator
\[
  \chi S_{\rm a}(0,\cdot) ( \tfrac{\I}{2} B^{-1}  c)
     \tilde{\Xi} T_M^*  
\]
acting on $f_M$, modulo a smoothing operator.

The operator $\chi S_{\rm a}(0,\cdot) ( \tfrac{\I}{2} B^{-1}  c)
     \tilde{\Xi}$ is a Fourier integral operator with canonical
relation
\[
  \big\{ ((\bfx,\bfxi), (\bfy(\bfx,\bfxi,t),t,
           \bfeta(\bfx,\bfxi,t),-c(\bfx) |\bfxi|)) \, \big|\ 
    |\eta_n(\bfx,\bfxi,t)| \ge \epsilon ,\ \epsilon > 0 \big\} .
\]
For an element $(\bfy',0,\bfeta',\omega)$
with $|\omega| > c |\bfeta' |$ there are two rays associated, namely with
$\eta_n = \pm \sqrt{ c^{-2} \omega^2 - | \bfeta' |^2 }$.
The $+$ sign propagates into $x_n < 0$ for decreasing time,
the $-$ sign points into $x_n >0$. The contributions are well separated
because of the cutoff for tangent rays present in $F_M$. 
Because of assumption (\ref{eq:assume_rays_one-to-one}) and the cutoff 
$\chi$, the contributions with $+$ sign can be ignored. We write 
$S_{\rm a}^{(-)}(0,\cdot) (-\tfrac{\I}{2} B^{-1}c) \tilde{\Xi} T_M^*$ for the Fourier integral
operator that propagates only the singularities from $M$ into the $x_n>0$ 
region for decreasing time.
By a similar reasoning as above, this is a Fourier integral operator 
with canonical relation contained in
\[
  \big\{ ((\bfx,\bfxi), (\bfy'(\bfx,\bfxi,t),t,\bfeta'(\bfx,\bfxi,t),-c(\bfx) |\bfxi|)) \, \big|\ 
    y_n(\bfx,\bfxi,t) = 0 \big\} .
\]
Again this is an invertible canonical relation.

The next step is the composition of the maps $(u_{0,\rm a},u_{0,\rm
  b}) \mapsto f_M$ and $f_M \mapsto (u_{\rmr,\rma}(\cdot,0),
u_{\rmr,\rmb}(\cdot,0))$.  As both maps are Fourier integral operators
with canonical relations that are the graph of an invertible map, the
composition is a (sum of) well defined Fourier integral operators.
The fields $u_{\rm a}$ and $u_{\rmr,\rma}$ are associated with negative
$\omega$, $u_{\rm b}$ and $u_{\rmr,\rmb}$ with positive $\omega$. One can verify
that the ``cross terms'' $u_{0,\rm a} \mapsto
u_{\rmr,\rmb}(\cdot,0)$ and $u_{0,\rmb} \mapsto
u_{\rmr,\rma}(\cdot,0)$ are smoothing operators. The maps $u_{0,\rm a}
\mapsto u_{\rmr,\rma}(\cdot,0)$ and $u_{0,\rm b} \mapsto
u_{\rmr,\rmb}(\cdot,0)$ are pseudodifferential operators. The principal symbol
$p_\rma(0;\bfx,\bfxi)$ is the product of $\Psi_{X_0}$ and another factor.

We proceed under the assumption that $\Psi_M(\bfx',t,\bfxi',t)$ is
supported in the region $0< t < t_1$, with $t_1$ sufficiently small
such that the explicit form of the Fourier integral operator can be
used. This assumption will be lifted at the end of the proof. We treat
only the map $u_{0,\rm a} \mapsto u_{\rmr,\rma}(\cdot,0)$, the map
$u_{0,\rm b} \mapsto u_{\rmr,\rmb}(\cdot,0)$ can be done in a similar
way. The map $u_{0,\rm a} \mapsto f_M$ can then be written in the form
\[
  f_{\rmr,\rma}(\bfy',0) = \frac{1}{(2\pi)^n} 
    \iint a^{\rm(fwd)}(\bfy',t,\bfx,\bfxi) \e^{\I ( \alpha(\bfy',0,t,\bfxi) - \bfx \cdot \bfxi)} 
      u_{0,a}(\bfx) \, d\bfxi \, d \bfx , 
\]
where the amplitude satisfies
\begin{equation} \label{eq:forw_ampl}
  a^{\rm (fwd)}(\bfy',t,\bfx,\bfxi) =
  -2 \I \chi(x_n)
  \omega \sqrt{1 - c(\bfy',0)^2 \omega^{-2} \bfeta'^2} 
 \sqrt{\det(\partial_\bfy\bfx)} \,
%\big|\pdpd{\bfy}{\bfx}\big|^{-\half}
  \Psi_M(\bfy,t,\bfeta', \omega) 
  \! \mod S^0(\Real^{2n} \times \Real^n) ,
\end{equation}
where $\omega = \partial_t \alpha = - c(\bfx) | \bfxi |$, $\bfeta =
\partial_\bfy \alpha$ and $\det(\partial_\bfy \bfx)$ is the
Jacobian of the ray flow as explained earlier. The {\em adjoint} of the
map $f_M \mapsto \chi u_{\rmr,\rma}(\cdot, 0)$ is given by
$\Xi^* T_M c  \frac{\I}{2} B^{-1} S_\rma(\cdot,0) \chi$, and is a Fourier 
integral operator with the same phase function 
$\alpha(\bfy',0,t,\bfxi) - \bfx \cdot \bfxi$ and amplitude
\begin{equation} \label{eq:backw_ampl}
  a^{\rm (bkd)}(\bfy',t,\bfz,\bfzeta)
  = \tfrac{\I}{2} \chi(z_n)  (- \omega^{-1}) c(\bfy) 
 \sqrt{\det(\partial_\bfy\bfx(\bfy',t,\bfzeta))} \,
%\big|\pdpd{\bfy}{\bfx}(\bfz,\bfzeta)\big|^{-\half}
  \Xi
    \! \mod S^{-2}(\Real^{2n} \times \Real^n) .
\end{equation}
The map $f_M \mapsto \chi u_{\rmr,\rma}(\cdot, 0)$ is therefore given by,
with the notation $\bfz$ instead of $\bfx \in \Real^n$,
\[
  u_{\rmr,\rma}(\bfz,0) = \frac{1}{(2\pi)^n}
    \iiint \overline{a^{\rm(bkd)}(\bfy',t,\bfz,\bfzeta)} 
      \e^{\I(- \alpha(\bfy',0,t,\bfzeta) + \bfz \cdot \bfzeta )} f_M(\bfy',t) 
    \, d\bfzeta \, d \bfy' \, dt .
\]
Therefore, the map $u_{0,\rm a} \mapsto \chi(z_n) u_{\rmr,\rma}(\cdot,0)$ has 
distribution kernel $K(\bfz,\bfx)$ given by
\begin{align} \label{eq:RT_sandw_phases_1}
  \frac{1}{(2\pi)^{2n}}
    \iiiint & \overline{a^{\rm (bkd)}(\bfy',t,\bfz,\bfzeta)} 
                a^{\rm (fwd)}(\bfy',t,\bfx,\bfxi) 
    \e^{\I(-\alpha(\bfy',0,t,\bfzeta)+\alpha(\bfy',0,t,\bfxi)+\bfz\cdot  \bfzeta-\bfx\cdot \bfxi)} 
    d\bfy' \, dt \, d\bfxi \, d\bfzeta .
\end{align}
Using a smooth cutoff the $(\bfxi,\bfzeta)$ integration domain can be
divided into three parts, one with $|\bfzeta| \le 2 | \bfxi |$, one
with $| \bfzeta | \ge \frac{4}{3} | \bfxi |$, and a third part
containing $(\bfzeta,\bfxi) = (\mathbf{0},\mathbf{0})$. In the first
part, the method of stationary phase can be applied to the integral
over $(\bfy',t,\bfzeta)$ using $| \bfxi |$ as large parameter.  We
show that there is a function $g(\bfz,\bfx,\bfxi)$ such that
\begin{equation} \label{eq:form_amplitude_result}
 \frac{1}{(2\pi)^n} 
  \iiint \overline{a^{\rm (bkd)}} a^{\rm (fwd)} 
    \e^{\I(- \alpha(\bfy',0,t,\bfzeta)+ \alpha(\bfy',0,t,\bfxi) + \bfz \cdot \bfzeta
    - \bfx \cdot \bfxi)} \, d\bfy' \, dt \, d \bfzeta
  = g(\bfz,\bfx,\bfxi) \e^{\I(\bfz-\bfx) \cdot \bfxi} ,
\end{equation}
and such that $g(\bfz,\bfx,\bfxi)$ is a symbol that has an asymptotic series 
expansion with leading order term satisfying
$g(\bfx,\bfx,\bfxi) = \Psi_{X_0}(\bfx,\bfxi)$. 

The first step in this computation is to determine the 
stationary points of the map
\[
  \Phi: (\bfy',t,\bfzeta) \mapsto 
    -\alpha(\bfy',0,t,\bfzeta)+\alpha(\bfy',0,t,\bfxi)+ \bfz\cdot  \bfzeta- \bfx\cdot \bfxi .
\]
By the properties of $\alpha$, $\pdpd{}{(\bfy',t)} \Phi = 0$
if and only if $\pdpd{\alpha}{(\bfy',t)}(\bfy',0,\bfzeta)
= \pdpd{\alpha}{(\bfy',t)}(\bfy',0,\bfxi)$
if and only if $(\bfy',0,t,\bfzeta)$ and $(\bfy',0,t,\bfxi)$ are associated
with the same bicharacteristic and hence $\bfzeta = \bfxi$.
Requiring that the derivative with respect to $\bfzeta$ is 0 gives that
\[
  - \partial_\bfxi \alpha (\bfy,t,\bfxi) + \bfz = 0 .
\]
Therefore, the bicharacteristic determined by $(\bfz,\bfxi)$ must be the same
as the bicharacteristic determined by $(\bfy',t,\bfzeta)$. Let 
$\psi(\bfy',t,\bfzeta; \bfx,\bfxi)$ be a $C^\infty$ cutoff function
that is one for a small neighborhood of $(\bfy',t,\bfzeta)$ around
the stationary value, and zero outside a slightly larger 
neighborhood. From the lemma of nonstationary phase one can
derive that the contribution to $g$ from the region away from the 
stationary point set is in $S^{-\infty}$.

At this point, observe that the second part, 
with $| \bfzeta | \ge \frac{4}{3} | \bfxi |$,
can be treated similarly, with the role of $\bfzeta$ and $\bfxi$
interchanged. In this case the stationary point set is in the region
where the amplitude is zero, and its contribution is of the form
(\ref{eq:form_amplitude_result}), but with $g$ in 
$S^{-\infty}(\Real^{2n} \times \Real^n)$. 
The third part, $\bfzeta,\bfxi$ around zero, also yields such a contribution 
with $g \in S^{-\infty}(\Real^{2n} \times \Real^n).$

To treat the case $(\bfy',t,\bfzeta)$ around the stationary point set, we apply 
a change of variables in the phase function. Setting $y_n=0$, it can be written as
\begin{align*}
  \alpha(\bfy,t,\bfxi) - \alpha(\bfy,t,\bfzeta)
  = {}&  \int_0^1 \frac{\partial}{\partial s} \alpha(\bfy,t,\bfzeta+s(\bfxi-\bfzeta)) \, ds
  = (\bfxi-\bfzeta) \cdot \int_0^1 
            \frac{\partial \alpha}{\partial \bfxi}(\bfy,t,\bfzeta+s(\bfxi-\bfzeta)) \, ds 
\\
  \stackrel{\rm def}{=}{}&
  (\bfxi - \bfzeta) \cdot X(\bfy',t,\bfzeta,\bfxi) .
\end{align*}
The goal is to rewrite (\ref{eq:RT_sandw_phases_1}) by change of variables 
into
\begin{equation} \label{eq:new_phase_function}
   \frac{1}{(2\pi)^n} 
  \iiint \psi\, \overline{a^{\rm (bkd)}} a^{\rm (fwd)} 
    \e^{\I((\bfxi -\bfzeta)\cdot X + \bfz\cdot \bfzeta - \bfx\cdot \bfxi)}
    \left|\frac{\partial X}{\partial(\bfy',t)}\right|^{-1} \, d\bfzeta \, dX ,
\end{equation}
so the next step is to prove that $\pdpd{X}{(\bfy',t)}$ is an
invertible matrix at the stationary points.  It is clear that, with $\bfxi=\bfzeta$ and $y_n=0$,
\[
  X(\bfy',t,\bfxi,\bfxi)
  = \frac{\partial \alpha}{\partial \bfxi}(\bfy,t,\bfxi)
  = \bfx(\bfy,t,\bfxi) ,
\]
where $(\bfy,t,\bfxi) \mapsto \bfx(\bfy,t,\bfxi)$ was discussed in 
subsection~\ref{subsec_FIO_S}.
The matrix $\pdpd{\bfx}{\bfy}$ is non-degenerate. Then we apply the implicit function theorem to the map $\bfx\mapsto(\bfy',t)$ obtained by setting $\bfy'=\bfy'(\bfx,\tilde t)$ in which $\tilde t$ is such that $y_n(\bfx,\tilde t)=0$, and use that there are no tangent rays, to obtain that the matrix $\pdpd{\bfx}{(\bfy',t)}$ has maximal rank at the stationary points, while the Jacobian satisfies
\[
  \left|\pdpdn{\bfx}{(\bfy',t)}\right|
  = 
  \left|\pdpdn{\bfx}{\bfy}\right| \, \left|\pdpdn{y_n}{t}\right| .
\]

The integral (\ref{eq:new_phase_function}) has a quadratic phase function 
$\bfzeta \cdot (X-\bfz)$, and can be performed as usual in the method of
stationary phase \cite[lemma 1.2.4]{Duistermaat1996}
This shows that $g(\bfz,\bfx,\bfxi)$ satisfies the symbol property.
Using (\ref{eq:forw_ampl}) and (\ref{eq:backw_ampl}) it follows that
\[
  g(\bfx,\bfx,\bfxi)
  =
  -2 \I (-c |\bfxi |) \sqrt{1 - c^2 \omega^{-2} \bfeta'^2} \left|\pdpdn{\bfx}{\bfy}\right|^{\half} \Psi_M
  \;\;
  (-\tfrac{\I}{2}) c(\bfx)  (c| \bfxi |)^{-1} c(\bfy)
  \left|\pdpdn{\bfx}{\bfy}\right|^{\half}
      \left|\pdpdn{\bfx}{\bfy}\right|^{-1} \left|\pdpdn{y_n}{t}\right|^{-1}
\]
Two terms need to be worked out, namely
$\sqrt{1 - c^2 \omega^{-2} \bfeta'^2} = \cos(\theta_M)$, in which $\theta_M$ is the angle of incidence
of a ray at $M$, and $\big|\pdpd{y_n}{t}\big| = c(\bfy) \cos(\theta_M)$.
Therefore indeed we have
\[
  g(\bfx,\bfx,\bfxi) =  \Psi_{X_0} \! \mod S^{-1}(\Real^{2n} \times \Real^n) .
\]
This concludes the proof of the small time result.

Next we extend this to the result for longer times. By a partition
of unity we can write $\Psi_M$ as a sum of terms with
$t \in [s,s+t_1]$ for some $s$. It is sufficient to prove the result
for each term, and we may therefore assume
$t \in [s,s+t_1]$ in the support of $\Psi_M$.
By a change of variable $t$ to $t-s$, it follows that 
\[
  P_{\rm a}(s) \stackrel{\rm def}{=}
  S_{\rm a}^{(-)}(s,\cdot) \left( - \tfrac{\I}{2} B^{-1} c \right)
   T_M^* F_M T_M
   S_{\rm a}(\cdot,s) \in \Op S^0(\Real^n)
\]
with principal symbol
\[
  \Psi_{X_s}(\bfx,\bfxi) = 
    \Psi_M (\bfy'(\bfx,\bfxi,t-s),t,\bfeta'(\bfx,\bfxi,t-s),-c(\bfx) | \bfxi |) ,
\qquad \text{with $t$ s.t. } y_n(\bfx,\bfxi,t-s) = 0 .
\]
From the group property of the $S_{\rm a}(t,s)$ it follows that
\[
  P_{\rm a}(0) = \chi S_{\rm a}(0,s) P_{\rm a}(s) S_{\rm a}(s,0) .
\]
The evolutions operators $S_{\rm a}(0,s)$, $S_{\rm a}(s,0)$ are each others
inverses. According to the Egorov theorem \cite[section 8.1]{Taylor81} the
operator $P_{\rm a}(0)$ is a pseudodifferential operator. For the symbol 
we find that it is given by $(\bfx,\bfxi)\mapsto\Psi_{X_s}(\bfy(\bfx,\bfxi,s),\bfeta(\bfx,\bfxi,s))$,
i.e.\ by $\Psi_{X_0}$. This completes the proof. 
\end{proof}

\section{Inverse scattering}    \label{sec_inverse_scattering}

This section deals with the inverse scattering problem. The diagram in
figure~\ref{RTM_diagram} shows how we theoretically approach RTM. The
forward modeling is given by $r\rightarrow u\rightarrow d$ in the
diagram. The \textsl{reflectivity} function~$r$ causes a scattered
wave field $u$, giving the \textsl{data} $d$ by restriction to the
surface $x_n=0$ (recall $\bfx'=[\bfx]_{1:n-1}$). The bottom line of
the diagram shows the inverse modeling. Data $d$ is propagated in
reverse time to the \textsl{reverse time continued} wave field $u_{\rm
  r}$. This wave field is mapped by the \textsl{imaging operator} $G$
to the \textsl{image} $i$. The \textsl{resolution operator} $R$ is the
map from the reflectivity to the image as result of the forward
modeling and the inversion. The scattering operator $F$ maps the
reflectivity to the continued scattered wave $u_\rmh$. As explained,
this field can be seen as the receiver wave field in an idealized
experiment. It contains all rays that are present in the scattered
wave, regardless whether they can be reconstructed by RTM. The
\textsl{revert operator} $P$ removes parts that are not present in the
receiver wave field. The field $u_\rmh$, central to the analysis, is
not actually computed.

We obtain the main result, the \textsl{imaging condition} (\ref{equ_Gomega_Omega}), in two steps. We propose the \textsl{imaging operator} $G$ and show in theorem \ref{thm_tildeR} and its proof that it is a FIO that maps the \textsl{reverse time continued wave field} to an image of the reflectivity. Hence it is an approximate inverse of the scattering operator. From this operator we subsequently derive an imaging condition in terms of solutions of partial differential equations, $g$ and $u_\rmr$. We first discuss a simplified case with constant coefficient.

Instead of condition (\ref{eq:assume_rays_one-to-one}) we have the
following condition for the RTM based inversion
\begin{equation} \label{eq:assume_bichar_no_return}
  \text{bicharacteristics that enter the region $x_n < 0$
           do not return to the region $x_n \ge 0$}.
\end{equation}
This will ensure that $u_{\rmr}$ is properly defined for the purpose of linearized inversion.  We also recall the assumption that there is no source wave field multipathing, formalized as the property (\ref{equ_SCA_TIC}). The assumption that there are no direct rays from the source to a receivers is incorporated in $P$, i.e.\ by means of $\Psi_M$, cf.\ (\ref{eq:prin_sym_F_M}).

\begin{figure}[ht]
   \begin{displaymath}
      \xymatrix{
           &  r(\bfx) \ar[ldd]_R \ar[d]^{F} \ar[r]  &  u(\bfx,t) \ar[dd]    \\
           &  u_{\rmh}(\bfx,t) \ar[d]^{P} \\
  i(\bfx)  &  u_\rmr(\bfx,t) \ar[l]^G                 &  d(\bfx',t) \ar[l]
      }
   \end{displaymath}
   \caption{Diagram showing the theoretical approach to RTM.}
   \label{RTM_diagram}
\end{figure}
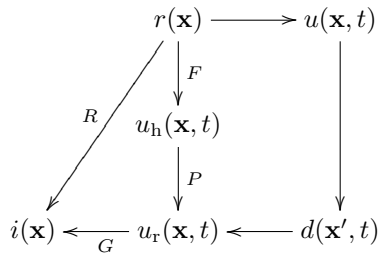

\newcommand{\waux}{w}

\subsection{Constant background velocity} \label{subsec_ISC} 

In this subsection we consider the case of constant background
velocity $c$ with a planar incoming wave, propagating in the positive
$x_3$ direction. The scattered field will be described by
\begin{equation} \label{eq:PDE_refl_field_3D}
  \left[ c^{-2} \partial_t^2 - \Delta \right] u (\bfx,t)
  = A \, \delta(t- c^{-1} x_3 ) r(\bfx) ,
\end{equation}
which is a slight simplification of (\ref{equ_SCA_scatteringproblem}). For simplicity the analysis will be 3-dimensional, but it applies to other dimensions as well.

The solution of the PDE (\ref{eq:PDE_refl_field_3D}) is given in the $(\bfxi,t)$ domain by
\begin{equation} \label{eq:sol_inhom_2_3D}
  \widehat{u}(\bfxi,t)
  = \int_0^t \left( \e^{\I c |\bfxi| (t-s)}
        - \e^{- \I c | \bfxi | (t-s)} \right)
    \frac{c^2}{2 \I c | \bfxi |} \widehat{f}(\bfxi,s) 
    \, ds ,
\end{equation}
where, for now, we denote by $f$ the right hand side of (\ref{eq:PDE_refl_field_3D}). The Fourier transform of $f$ is hence needed. Let $\widetilde{r}(\xi_1,\xi_2,x_3)$ be the Fourier transform of $r$ with respect to $(x_1,x_2)$ but not $x_3$. The Fourier transform of $A \delta( t - \frac{x_3}{c} ) r(\bfx)$ is given by
\begin{equation}
\label{eq:FT_refl_source_3D}
  \int \e^{-\I x_3 \xi_3} A \delta(t - \frac{x_3}{c}) \widetilde{r}(\xi_1,\xi_2,x_3) \, dx_3 
  = c A \e^{-\I \xi_3 c t } \tilde{r}(\xi_1,\xi_2, c t) .
\end{equation}
Next we use (\ref{eq:sol_inhom_2_3D}) and (\ref{eq:FT_refl_source_3D}) to solve
(\ref{eq:PDE_refl_field_3D}), and we make a change of variable $c s = \tilde{z}$. This yields
\[
  \widehat{u}(\bfxi,t)
  = \int_0^{t c} \left( \e^{\I |\bfxi|(c t-\tilde{z})}
        - \e^{- \I |\bfxi|(c t-\tilde{z})} \right)
    \frac{c^2}{2 \I c |\bfxi|} A
        \e^{-\I \xi_3 \tilde{z}} \tilde{r}(\xi_1,\xi_2,\tilde{z}) \, d\tilde{z} .
\]
We can recognize in this formula a Fourier transformation with respect to $\tilde{z}$. However, the Fourier transform of $r$ is not evaluated at $\xi_3$, but at $\xi_3 \pm |\bfxi|$, because $\tilde{z}$ occurs at several places in the complex exponents. Under the assumption that the support of $r$ is contained in $0 < x_3 < c t$ (in other words, that we consider the field at time $t$ such that the incoming wave front has completely passed the support of the reflectivity), the formula equals
\begin{equation} \label{eq:field_from_r_3D}
  \widehat{u}(\bfxi,t) 
  = \e^{\I |\bfxi|ct} 
        \frac{c^2 A}{2 \I c |\bfxi|} 
        \widehat{r}(\bfxi+(0,0,|\bfxi|))
    - \e^{- \I |\bfxi|ct} 
        \frac{c^2 A}{2 \I c |\bfxi|} 
        \widehat{r}(\bfxi-(0,0,|\bfxi|))
\end{equation}
The field in position coordinates is given by the inverse Fourier
transform of this, i.e.\ by
\begin{multline} \label{eq:u_t_x_z_3D}
  u(\bfx,t) 
  = \frac{1}{(2\pi)^3} \int_{\Real^3} \left[
    \e^{\I |\bfxi|ct} 
        \frac{c^2 A}{2 \I c |\bfxi|} 
        \widehat{r}(\bfxi+(0,0,|\bfxi|))
    - \e^{- \I |\bfxi|ct} 
        \frac{c^2 A}{2 \I c |\bfxi|} 
        \widehat{r}(\bfxi-(0,0,|\bfxi|)) \right]
    \e^{\I \bfx \cdot \bfxi}
    \, d \bfxi
\end{multline}
The two terms yield complex conjugate contributions after integration.
To see this, change the integration variables in the second term to 
$-\bfxi$, and use that the property that $r(\bfx)$ is real for 
all $\bfx$ is equivalent to
$\widehat{r}(\bfxi) = \overline{\widehat{r}(-\bfxi)}$
for all $\bfxi$. Therefore
\begin{equation} \label{eq:short_u_t_x_z}
  u(\bfx,t) = \frac{1}{(2\pi)^3} \operatorname{Re} \int_{\Real^3}
    \e^{- \I |\bfxi|ct +\I \bfx\cdot\bfxi}
        \frac{\I c A}{|\bfxi|} 
        \widehat{r}(\bfxi-(0,0,|\bfxi|)) \, d\bfxi .
\end{equation}

There are three wave vectors in (\ref{eq:short_u_t_x_z}), $\bfxi$ is the wave vector of the outgoing reflected wave, $(0,0,|\bfxi|)$ can be interpreted as the wave vector of the incoming wave, while $\bfxi - (0,0,|\bfxi|)$ can be interpreted as the reflectivity wavenumber, which, for a conormal singularity for example, would be normal to the reflector. 

In this simplified analysis we assume that the reverse time continued receiver field $u_{\rm r}$ satisfies a homogeneous wave equation with equal final values (after the scattering) as $u$, like $u_\rmh$ in (\ref{equ_SCA_continuedscattered}), i.e.\ it results from an idealized experiment as explained in section~\ref{subsec_continued_scat_wave}. This means that $u_{\rm r}$ is also given by (\ref{eq:short_u_t_x_z}), except that this formula is now valid for all $t$. 

The basic idea of imaging is to time-correlate the source field with the
receiver field. Approximating the source field by $A \delta(t-x_3/c)$
this becomes evaluating the receiver field at the arrival time of the
incoming wave and multiplication by $A$. Hence, a first guess for the
image would be $I_0 = A u(\bfx,x_3/c)$. This, however will not yield an 
inverse. Using some advance knowledge we will define instead as our 
image
\begin{equation} \label{eq:recon_form}
  I(\bfx) = \frac{2}{c^2 A} (\partial_t + c \partial_{x_3}) u (\bfx,x_3/c) .
\end{equation}
We have from (\ref{eq:short_u_t_x_z})
\begin{equation}
  \frac{2}{c^2 A} (\partial_t+c \partial_{x_3}) u(\bfx,t) 
  = \frac{2}{(2\pi)^3} \operatorname{Re} \int_{\Real^3}
    \left(1 - \frac{\xi_3}{|\bfxi|} \right)
    \e^{-\I |\bfxi|ct + \I \bfx\cdot\bfxi}
     \, \widehat{r}(\bfxi-(0,0,|\bfxi|)) \, d\bfxi .
\end{equation}
Setting $t=x_3/c$ we find
\begin{equation} \label{eq:recon_2_3D}
  I(\bfx)
  = \frac{2}{(2\pi)^3} \operatorname{Re} \int_{\Real^3}
    \left(1 -  \frac{\xi_3}{|\bfxi|} \right)
    \e^{\I \bfx \cdot (\bfxi-(0,0,|\bfxi|)}
    \, \widehat{r}(\bfxi-(0,0,|\bfxi|)) \, d\bfxi  .
\end{equation}
We carry out a coordinate transformation,
\begin{equation} \label{eq:simpliof_transform}
  \tilde{\bfxi} = \bfxi - (0,0,|\bfxi|), 
\qquad\qquad
  \left| \frac{\partial \tilde{\bfxi}}{\partial\bfxi} \right|
  = 1 - \frac{\xi_3}{|\bfxi|} .
\end{equation}
The image of this transformation is the halfplane $\tilde{\xi}_3 < 0$,
while the Jacobian is as given in (\ref{eq:simpliof_transform}), and
exactly equals the factor $1 - \frac{\xi_3}{|\bfxi|}$ from
the derivative operator $\partial_t + c \partial_{x_3}$.
Therefore by a change of variables (\ref{eq:recon_2_3D}) equals
$\frac{1}{(2\pi)^3} \operatorname{Re} \int_{\tilde{\xi}_3<0}
    \e^{\I\bfx \cdot \tilde{\bfxi}} \widehat{r}(\tilde{\bfxi}) \, d\tilde{\bfxi}$.
This can be rewritten  as
\begin{equation}
  I(\bfx) 
  = \frac{1}{(2\pi)^3} \int_{\tilde{\xi}_3 \neq 0}
    \e^{\I\bfx \cdot \tilde{\bfxi} } \, \widehat{r}(\tilde{\bfxi}) \, d \tilde{\bfxi} .
\end{equation}
The right-hand side is almost the inverse Fourier transform, except for the exclusion of the set $\tilde{\xi}_3 = 0$ from the integration domain. This expresses the difficulty with inverting from direct waves. This simple calculation gives the motivation for the imaging condition (\ref{equ_Gomega_Omega}) below, in particular, for the term involving the gradient $\partial_\bfx \widehat{u}_\rmr(\bfx,\omega)$.

\subsection{Imaging condition}  \label{subsec_IC}
We present the main result of the paper. The \textsl{imaging condition} yields a mapping of the source wave $g(\bfx,t)$ and the reverse time continued wave $u_{\rm r}(\bfx,t)$ to an image $i(\bfx)$ of the reflectivity. 
We will show that the following imaging condition yields a partial inverse,
\begin{equation}    \label{equ_Gomega_Omega}
    i(\bfx) = \frac{1}{2\pi} \int \frac{\Omega(\omega)}
                {\I\omega|\widehat{g}(\bfx,\omega)|^2} 
    \left(\overline{\widehat{g}(\bfx,\omega)}\widehat{u}_{\rm r}(\bfx,\omega) 
    - \frac{c(\bfx)^2}{\omega^2}\, \partial_\bfx\overline{\widehat{g}(\bfx,\omega)}\cdot\partial_\bfx\widehat{u}_{\rm r}(\bfx,\omega)
    \right) \, d\omega.
\end{equation}
Here $\Omega(\omega)$ is a smooth function, valued 0 on a bounded neighborhood of the origin, and 1 outside a slightly larger neighborhood. These neighborhoods are obtained in the proof of the theorem.

To characterize $i(\bfx)$, the relation (\ref{equ_SCA_coivar}) between $\bfzeta$ and $\bfxi$ is important. We observe that the inverse function $\bfxi(\bfzeta)$ of (\ref{equ_SCA_coivar}) is defined on the halfspace 
\begin{equation} \label{eq:halfspace_zeta}
  \{ \bfzeta \in \Real ^n \backslash 0 \, |\, \bfzeta \cdot \bfns(\bfz) < 0 \} .
\end{equation}
The function $p_{\rm a}(T_{\rm s}(\bfz);\bfz,\bfxi(\bfzeta))$, $p_\rma$ the principal symbol of the revert operator, is in principle defined only on (\ref{eq:halfspace_zeta}). However, due to the DSE, it is zero for $\bfzeta$ near the boundary of this halfspace and we will consider it as a function on $\Real^n \backslash 0$ that is zero outside (\ref{eq:halfspace_zeta}). With this definition, the function $(\bfz,\bfzeta) \mapsto p_{\rm a}(T_{\rm s}(\bfz);\bfz,\bfxi(\bfzeta))$ is an order 0 symbol.

\begin{thm} \label{thm_R}
Let image $i(\bfx)$ be defined by (\ref{equ_Gomega_Omega}), and assume
(\ref{eq:assume_SME}), (\ref{eq:assume_supp_r}) and (\ref{eq:assume_bichar_no_return}). Define operator $R$ by the map from the reflectivity $r$ to the image, $R r(\bfx)=i(\bfx)$. Then $R$ is a pseudodifferential operator of order zero, and its principal symbol satisfies
\begin{equation} \label{eq:prin_symb_R}
  \text{p.s.}(R)(\bfz,\bfzeta)
  = p_{\rm a}(T_{\rm s}(\bfz);\bfz,\bfxi(\bfzeta))
    + p_{\rm a}(T_{\rm s}(\bfz);\bfz,\bfxi(-\bfzeta)) ,
\end{equation}
where the map $(\bfz,\bfzeta) \mapsto p_{\rm a}(T_{\rm s}(\bfz);\bfz,\bfxi(\bfzeta))$ is as just described.
\end{thm}

Operator $R$ will be referred to as the \textsl{resolution operator}. From the proof of the result it can be seen that the first contribution on the right-hand side of (\ref{eq:prin_symb_R}) corresponds to the negative frequencies and the second contribution to the positive frequencies. As the supports, i.e.\ (\ref{eq:halfspace_zeta}) for the first, of these two terms are disjoint, (\ref{eq:prin_symb_R}) defines a symbol that is one on a subset of $\Real^n \times \Real^n \backslash 0$. Hence, the map $d \mapsto i$ given by (\ref{equ_Gomega_Omega}) can rightfully be called a partial inverse.

The imaging condition (\ref{equ_Gomega_Omega}) is based on the actual source field $g$. Before proving theorem~\ref{thm_R}, we derive an intermediate result with an imaging condition based on the source wave traveltime $T_{\rm s}(\bfx)$, and the highest order contribution to the amplitude $A_{\rm s}(\bfx)$. Let $\waux\in\mathcal{E}'(Y\tines\Real)$ be an auxiliary distribution. Let operators $H$ and $K$ be defined by
\begin{equation}
\begin{split} \label{equ_ISV_G}
   H\waux(\bfy,t)   &= \frac{1}{A_{\rm s}(\bfy)}\partial_t^{-\!\frac{n+1}{2}} [\partial_t+c(\bfy)\bfns(\bfy)\cdot\partial_\bfy]\waux(\bfy,t), \\
   K\waux(\bfz)     &= \waux(\bfz,T_{\rm s}(\bfz)).
       \end{split}
\end{equation}
Operator $K$ is a restriction to a hypersurface in $\Real^{n+1}$. Operator $H$ is a pseudodifferential operator. Operator $\partial_t^{-\!\frac{n+1}{2}}$ is to be read as the pseudodifferential operator with symbol $\omega\mapsto\tilde\sigma(\omega)(\I\omega)^{-\!\frac{n+1}{2}}$ in which $\tilde\sigma$ is a smooth function, valued 1 except for the origin where it is 0.
Because $P$ and $F$ are defined as matrix operators, we define $V_1 = \begin{pmatrix} 1 & 0 \end{pmatrix}$ which projects out the first component of a two-vector. We define the \textsl{imaging operator} \mbox{$G=KH$}. 
\begin{thm} \label{thm_tildeR}
If  (\ref{eq:assume_SME}), (\ref{eq:assume_supp_r}) and (\ref{eq:assume_bichar_no_return}) are satisfied and $\widetilde{R}$ is given by
\begin{equation} \label{equ_ISV_tildeR}
  \widetilde{R}r = Gu_\rmr = G V_1 P F r ,
\end{equation}
then $\widetilde{R}$ is a pseudodifferential operator of order zero with principal symbol given by (\ref{eq:prin_symb_R}). 
\end{thm}

\begin{proof}
We first work out the details for the negative frequencies, leading to a characterization of $\widetilde{R}_\rma=KHcP_\rma F_\rma$. We then consider the positive, and add the contributions, $\widetilde{R}=\widetilde{R}_\rma+\widetilde{R}_\rmb$.

(i) We show that the composition $\widetilde{R}_\rma=KHcP_\rma F_\rma$ is a FIO and that it is microlocal, i.e.\ has canonical relation that is a subset of the identity. The kernel of operator $K$ is an oscillatory integral,
\begin{equation}
   K\waux(\bfz) = (2\pi)^{-n-1}\iint
                              \e^{\I\bfeta\cdot(\bfz-\bfy)+\I\omega(T_{\rm s}(\bfz)-t)}\, \waux(\bfy,t)\, d(\bfy,t) \, d\bfeta \, d\omega
\end{equation}
with canonical relation
\begin{multline}  \label{equ_ISC_canrel_K}
   \Upsilon = \big\{ ((\bfz,\bftheta),(\bfy,t,\bfeta,\omega))\ \big|\
                     (\bfy,\bfeta)\in\cotbun{Y},\ t=T_{\rm s}(\bfy),\ \omega\in\Real\!\setminus\!\{0\},\ \\
                     \bfz=\bfy,\ \bftheta = \bfeta+c(\bfy)^{-1}\omega\,\bfns(\bfy)\, \big\} .
\end{multline}
First consider $K\pi_\rma F_\rma$, which is the composition of $K$ and $\pi_\rma F_\rma$ with canonical relations given respectively by $\Upsilon$ (\ref{equ_ISC_canrel_K}) and $\Lambda$ (\ref{equ_SCA_canrel_global}). We consider the composition of the Fourier integrals $K$ and $\pi_\rma F_\rma$, using the composition theorem based on the canonical relations, see \cite[Theorem 2.4.1]{Duistermaat1996} or \cite{Treves80v2}.
Let $((\bfz,\bftheta),(\bfx,\bfzeta))\in\Upsilon\circ\Lambda$ then there exist a \mbox{$(\bfy,\bfeta)\in\cotbun{Y}$} that is not in $V_{\rms,t}$, time $t=T_{\rm s}(\bfy)$ and $\omega=-c(\bfy)|\bfeta|$ such that \mbox{$((\bfz,\bftheta),(\bfy,t,\bfeta,\omega))\in\Upsilon$} and \mbox{$((\bfy,t,\bfeta,\omega),(\bfx,\bfzeta))\in\Lambda$}. As a result one has $(\bfx,\bfxi)=\Phi_{T_{\rm s}(\bfx)-T_{\rm s}(\bfy)}(\bfy,\bfeta)$, which means that $\bfx$ and $\bfy$ are on the same ray and separated in time by $T_{\rm s}(\bfy)-T_{\rm s}(\bfx)$. Condition (\ref{equ_SCA_TIC}) (SME) now implies that this ray must coincide with a source ray. As source rays are excluded, i.e.\ $(\bfy,\bfeta)\notin V_{\rms,t}$, the only possibility is that $\bfx=\bfy$. The conclusion is that $(\bfz,\bftheta)=(\bfx,\bfzeta)$.

It is straightforward to establish that the composition of canonical relations is transversal, and that the additional conditions of the composition theorem of FIOs are satisfied. Hence $K \pi_\rma F_\rma$ is a FIO with canonical relation contained in the identity. The operators $H$ and $P_\rma$ are pseudo\-differential operators, and $\pi_\rma$ and $P_\rma$ can be constructed such that \mbox{$\wfs(P_\rma w)\subset\wfs(\pi_\rma w)$} for all~$w$. The conclusion is that $\widetilde{R}_\rma=KHcP_\rma F_\rma$ is a FIO with identity canonical relation, and hence a pseudodifferential operator.

\smallskip

(ii) We show that $\widetilde{R}=KHV_1PF$ is a pseudodifferential operator that can be written as the integral (\ref{equ_tildeR_integral}) below. For $F$ we use the local expressions (\ref{equ_SCA_Fa}).
Because $P$ is a $t$-family of pseudodifferential operators and $F_{\rma}\rho$ is a $t$-family of FIOs, the composition $P_\rma(t)F_\rma(t)\rho$ is a FIO with phase inherited from $F_\rma(t)\rho$, i.e.\ $\phiT$. The highest order contribution to its amplitude is $p_\rma(t;\bfy,\partial_\bfy{\phiT})\mathrm{A_F}$. The composition with $H$ can be done similarly, because $P_\rma(t)F_\rma(t)\rho$ can also be viewed as a FIOs with output variables $(\bfy,t)$. In this proof we will denote the highest order contribution to the amplitude of $H c P_\rma(t)F_\rma(t)\rho$ by $\mathrm{A_{HPF}}(\bfy,t_1,\bfx,\bfxi)$. It can be written in the form
\begin{equation}  \label{equ_ISV_ampHF}
   \mathrm{A_{HPF}}(\bfy,t_1,\bfx,\bfxi) =
   \left(1+\frac{c(\bfy)\bfns(\bfy)\cdot\partial_\bfy{\alpha}}{\partial_t{\alpha}}\right) \, 
   \frac{2 \I c(\bfy) p_\rma(t_1;\bfy,\partial_\bfy{\alpha}) {a} A_{\rm s}(\bfx) \partial_t{\alpha}}{c(\bfx) A_{\rm s}(\bfy)}.
\end{equation}
For all occurences of ${\alpha}$ and ${a}$ the arguments are $(\bfy,t_1-T_{\rm s}(\bfx);\bfxi)$.

Next we consider the applicaton of the restriction operator $K$. We have already argued that $\widetilde{R}_\rma $ is a FIO with canonical relation contained in the identity. This implies that, to prove the theorem, it is sufficient to do a local analysis using (\ref{equ_SCA_Fa}). The local analysis shows again that $\widetilde{R}_\rma $ is a pseudodifferential operator, but also gives the required explicit formula for the amplitude. 

The local phase function of $K H c P_\rma(t)F_\rma(t)\rho$ will be denoted by $\psi(\bfz,\bfx,\bfxi)$. Applying $K$ to $\phiT$, i.e.\ setting $t=T_\rms(\bfz)$, yields
\begin{equation}  \label{equ_ISV_psi}
   \psi(\bfz,\bfx,\bfxi) = {\alpha}(\bfz,T_{\rm s}(\bfz)-T_{\rm s}(\bfx);\bfxi) - \bfxi\cdot\bfx.
\end{equation}
The stationary point set of $\psi$, denoted by $\Psi$, is given by the triplets $(\bfz,\bfx,\bfxi)$ that solve
\begin{equation}  \label{equ_ISV_statsetpsi}
   \partial_\bfxi{\alpha}(\bfz,T_{\rm s}(\bfz)-T_{\rm s}(\bfx);\bfxi) = \bfx.
\end{equation}
The interpretation of $(\bfz,\bfx,\bfxi)\in\Psi$ is that a ray with initial condition $(\bfx,\bfxi)$ arrives at $\bfz$ after time lapse $T_{\rm s}(\bfz)-T_{\rm s}(\bfx)$. Application of the SME and the DSE now implies that $\bfz=\bfx$.

Below we will define a transformation of covariables. To prepare for this, we introduce a smooth cutoff function $\chi:Z\tines X\tines\Conic\rightarrow\Real$ accordingly. A Fourier integral may be restricted to a neighborhood of the stationary point set at the expense of a regularizing operator. Therefore, $\chi(\bfz,\bfx,\bfxi)$ is set to 1 in the neighborhood of $\Psi$ and 0 elsewhere. This means that $\bfx$ is close to $\bfz$ in $\supp(\chi)$. The second issue is related to the DSE, which is required for the definition of the transformation. The cutoff $\chi$ is assumed to also remove singularities on a neighborhood of the direct rays. We set $\chi(\bfz,\bfx,\bfxi)$ to 0 if $\bfxi$ lies within a narrow conic set with solid angle $\Omega(\bfz)$ around the principal direction $\bfns(\bfz)$. The solid angle $\Omega(\bfz)$ will be discussed later. We can hence write 
\begin{equation} \label{equ_ISV_tildeRa_xi}
   \widetilde{R}_\rma r(\bfz) = (2\pi)^{-n} \iint \e^{\I\psi(\bfz,\bfx,\bfxi)}
   \chi(\bfz,\bfx,\bfxi)\, \mathrm{A_{HPF}}(\bfz,T_{\rm s}(\bfz),\bfx,\bfxi) d\bfxi\, r(\bfx)\, d\bfx .
\end{equation}
in which, of course, the integration domain is implicitly restricted to $\supp(\chi)$.
 
Next we introduce covariable $\bftheta$ to transform phase $\psi$ into the form $\bftheta\cdot(\bfz-\bfx)$. By definition $\bftheta(\bfz,\bfx,\bfxi)=-\int_0^1\partial_x\psi(\bfz,\tilde\bfx(\mu),\bfxi)d\mu$ in which \hbox{$\tilde\bfx(\mu)=\bfz+\mu(\bfx-\bfz)$}. The phase function now transforms into
\begin{equation}  \label{equ_ISV_quadraticphase}
   \psi(\bfz,\bfx,\bfxi) = \psi(\bfz,\bfz,\bfxi) + \int_0^1 \partial_\mu[\psi(\bfz,\tilde\bfx(\mu),\bfxi)] d\mu = \bftheta(\bfz,\bfx,\bfxi) \cdot (\bfz-\bfx).
\end{equation}
To better understand the transformation and to determine the new domain of integration, i.e.\ $\bftheta(\supp(\chi))$, and the Jacobian we apply the chain rule to the definition of $\psi$. This leads to 
\begin{equation*}
    \bftheta(\bfz,\bfx,\bfxi) = \bfxi + \int_0^1 \partial_t{\alpha}(\bfz,T_{\rm s}(\bfz)-T_{\rm s}(\tilde\bfx);\bfxi)\, \partial_\bfx T_{\rm s}(\tilde\bfx)\, d\mu.
\end{equation*}
There exists an $\check\bfx$ such that $(\bfz,\check\bfx,\bfxi)\in\Gamma_{T_{\rm s}(\bfz)-T_{\rm s}(\tilde\bfx)}$, i.e.\ $\check\bfx$ and $\bfz$ are connected by a ray. Note that $\check\bfx=\bfx_\Gamma(\bfz,T_{\rm s}(\bfz)-T_{\rm s}(\tilde\bfx),\bfxi)$ will do, see section \ref{subsec_FIO_S} for notation $\bfx_\Gamma$. By using the identities $\partial_t{\alpha}=-c(\check\bfx)|\bfxi|$ and $c(\tilde\bfx)\partial_\bfx T_{\rm s}(\tilde\bfx)=\bfns(\tilde\bfx)$, one gets
\begin{equation}  \label{equ_ISV_thetaofxi}
    \bftheta(\bfz,\bfx,\bfxi) = \bfxi - |\bfxi|\,\bfn(\bfz,\bfx,\bfxi) \quad\mathrm{with}\quad
    \bfn(\bfz,\bfx,\bfxi) = \int_0^1 \frac{c(\check\bfx)}{c(\tilde\bfx)}\bfns(\tilde\bfx)\, d\mu.
\end{equation}
The Jacobian now follows from this result. By an easily verified calculation, one finds 
\begin{equation} \label{equ_ISV_Jacobian}
  \left|\partial_\bfxi\bftheta\right| 
  = \left| \det\left({\rm I_n} 
        - \frac{\bfxi}{|\bfxi|}\otimes\bfn(\bfz,\bfx,\bfxi)\right) \right| 
  = \left| 1 - \frac{\bfxi}{|\bfxi|}\cdot\bfn(\bfz,\bfx,\bfxi) \right|.
\end{equation}

With these formulae at hand a sensible choice can be made for the solid angle $\Omega(\bfz)$. The angle must be large enough to meet the following inequality for all elements of $\supp(\chi)$:
\begin{equation} \label{equ_ISV_inequality}
   |\bfxi\cdot\bfn(\bfz,\bfx,\bfxi)| \, < \, |\bfxi|\min\{1,|\bfn(\bfz,\bfx,\bfxi)|^2\}. 
\end{equation}
We will now give the motivation. For $\bfxi\mapsto\bftheta(\bfz,\bfx,\bfxi)$ to be injective, given $(\bfz,\bfx)$, the Jacobian must be nonzero. This is true due to the inequality, which is nontrivial if $|\bfn|>1$. This affirms the local invertibility, and an easy exercise proofs its injectivity. A second argument concerns the domain of integration $\bftheta(\supp(\chi))$. The inequality guarantees that $\bftheta(\bfz,\bfx,\bfxi)\cdot\bfn(\bfz,\bfx,\bfxi)<0$ for all points in $\supp(\chi)$, which is nontrivial if $|\bfn|<1$. This fact will play a role in gluing $\widetilde{R}_\rma r$ and $\widetilde{R}_\rmb r$ together, which will be done in following paragraphs. Because $\bfx$ is in the neighborhood of $\bfz$, so are $\tilde\bfx$ and~$\check\bfx$. This implies that $\bfn(\bfz,\bfx,\bfxi)$ is close to $\bfns(\bfz)$, and $|\bfn|\approx 1$. This is as close as needed by narrowing the spatial part of the cutoff function $\chi$ around the diagonal of $Z\tines X$.

By using the new variable $\widetilde{R}_\rma r$'s integral expression (\ref{equ_ISV_tildeRa_xi}) transforms into 
\begin{equation}  \label{equ_tildeR_integral0}
   \widetilde{R}_\rma r(\bfz) = (2\pi)^{-n} \iint_{\bftheta(\supp(\chi))} \mathrm{A_{\widetilde{R}}}(\bfz,\bfx,\bftheta)
  \e^{\I\bftheta\cdot(\bfz-\bfx)}
  d\bftheta\, r(\bfx)\, d\bfx,
\end{equation}
where we define
\begin{equation} \label{eq:define_AR_threearg}
  \mathrm{A_{\widetilde{R}}}(\bfz,\bfx,\bftheta)
  = |\partial_\bfxi\bftheta|^{-1} 
        \chi(\bfz,\bfx,\bfxi)\, \mathrm{A_{HPF}}(\bfz,T_{\rm s}(\bfz),\bfx,\bfxi) .
\end{equation}
Concerning the integration domain it can be observed that, for a given $(\bfz,\bfx)$ the set $\bftheta(\supp(\chi))$ is contained in the halfspace $\{\bftheta\in\Conic\, |\, \bftheta\cdot\bfn<0\}$. 

\smallskip

(iv) While the expression (\ref{equ_tildeR_integral0}) defines a pseudodifferential operator of order 0, it is given in a non-standard from. It differs from a regular pseudodifferential operator, because the the amplitude $\mathrm{A_{\widetilde{R}}}(\bfz,\bfx,\bftheta)$ depends on $(\bfz,\bfx,\bftheta)$ and not only on $(\bfz,\bftheta)$. Another amplitude
that does not depend on $\bfx$ can be found by
\begin{equation} \label{equ_other_amplitude}
   \mathrm{A_{\widetilde{R}}}(\bfz,\bfz,\bftheta) + \sum_{k=1}^n \int_0^1 D_{\theta_k} \partial_{x_k} \mathrm{A_{\widetilde{R}}}(\bfz,\bfz+\mu(\bfx-\bfz),\bftheta)\, d\mu,
\end{equation}
which is an application of formulae (4.8)-(4.10) of Treves \cite{Treves80v1}. The first term is the principal symbol of $\widetilde{R}_\rma$, which has order 0. The second term in (\ref{equ_other_amplitude}) does not contribute to the principal part, it corresponds to a pseudodifferential operator of order $-1$. We will denote by $\mathrm{A_{\widetilde{R}}}(\bfz,\bftheta)$ (with two arguments) the symbol of $\widetilde{R}$. 

To evaluate of $\mathrm{A_{HPF}}$ (\ref{equ_ISV_ampHF}) on the diagonal one applies (\ref{equ_IVPFIO_inipha}), the relation $\partial_t{\alpha}(\bfz,0;\bfxi)=-c(\bfz)|\bfxi|$ for the phase and the result ${a}(\bfz,0;\bfxi)=\tfrac{\I}{2c(\bfz)|\bfxi|}$ for the amplitude. This yields
\begin{equation} \label{eq:A_HPF_formula}
   \mathrm{A_{HPF}}(\bfz,T_{\rm s}(\bfz),\bfz,\bfxi)
   = \left(1 - \frac{\bfns(\bfz)\cdot\bfxi}{|\bfxi|} \right)
              p_\rma(T_\rms(\bfz);\bfz,\bfxi)
   = |\partial_\bfxi\bftheta| \, p_\rma(T_\rms(\bfz);\bfz,\bfxi) ,
\end{equation}
see also (\ref{equ_ISV_Jacobian}). In view of (\ref{eq:define_AR_threearg})-(\ref{eq:A_HPF_formula}), we have $\text{p.s.}(\mathrm{A_{\widetilde{R}}})(\bfz,\bftheta) = \chi(\bfz,\bfz,\bfxi)\, p_{\rm a}(T_{\rm s}(\bfz);\bfz,\bfxi)$. Note that $\bfeta=\partial_\bfy{\alpha}=\bfxi$ holds on the diagonal, and that $\bfxi=\bfxi(\bftheta)$.

We now come back to the formal role of cutoff function $\chi$. By requiring $\chi(\bfz,\bfz,\bfxi)=1$ on $\supp(p_{\rm a}(T_{\rm s}(\bfz);\bfz,\bfxi))$ the cutoff function can be left out. This requirement is allowed because $\Omega(\bfz)$ in the construction of $\chi$ can be chosen arbitrarily tight by narrowing the spatial support of $\chi$ around the diagonal. Therefore
\begin{equation} \label{eq:ampl_A_tildeR}
  \text{p.s.}(\mathrm{A_{\widetilde{R}}})(\bfz,\bftheta) 
  = p_{\rm a}(T_{\rm s}(\bfz);\bfz,\bfxi(\bftheta)) .
\end{equation}

\smallskip

(v) A key step is the inclusion of both negative and positive frequencies. In section~\ref{sec_solution_ivp} we saw that ${a}(\bfx,t;\bfxi)\, \e^{\I\lambda{\alpha}(\bfx,t;\bfxi)}$ and ${b}(\bfx,t;- \bfxi)\, \e^{\I\lambda{\beta}(\bfx,t;- \bfxi)}$ have a symmetry relation: They yield complex conjugate contributions (note the $-$ signs). The consequences of this property can be traced through this proof. We find that
$\widetilde{R}_\rmb r(\bfz) = 
  (2\pi)^{-n} \iint_{-\bftheta(\supp(\chi))} 
  \overline{\mathrm{A_{\widetilde{R}}}(\bfz,-\bftheta)}
  \e^{\I\bftheta\cdot(\bfz-\bfx)} d\bftheta\, r(\bfx)\, d\bfx $,
and consequently, modulo a regularizing contribution,
\begin{equation}  \label{equ_tildeR_integral}
   \widetilde{R} r(\bfz) 
   = (2\pi)^{-n} \iint
    \left[ \mathrm{A_{\widetilde{R}}}(\bfz,\bftheta) +
    \overline{\mathrm{A_{\widetilde{R}}}(\bfz,-\bftheta)} \right]
    \e^{\I\bftheta\cdot(\bfz-\bfx)}
    d\bftheta\, r(\bfx)\, d\bfx  .
\end{equation}
The $\bftheta$-integration is over the full space because the definition of $\mathrm{A_{\widetilde{R}}}(\bfz,\bftheta)$ can be smoothly extended such that it is zero outside the domain $\bftheta(\supp(\chi))$. In view of (\ref{eq:ampl_A_tildeR}) this proves the claim.
\end{proof}

%{\bf INTRODUCTORY PARAGRAPH PROOF OF THEOREM~\ref{thm_tildeR} }

\begin{proof}[Proof of theorem~\ref{thm_R}]
The first step in deriving the imaging condition is to rewrite operators $H$, $K$ and $G$ (\ref{equ_ISV_G}). Let $\waux(\bfx,t)$ again be an auxiliary distribution. In this section $\widehat\waux(\bfx,\omega)$ will denote its temporal Fourier transform. Because $\waux(\bfx,t)=\frac{1}{2\pi}\int\e^{\I\omega t}\widehat{\waux}(\bfx,\omega)\, d\omega$, one has
\begin{equation}
\begin{split}   \label{equ_HKomega}
    \widehat{H\waux}(\bfx,\omega) &= 
                                    \frac{\tilde\sigma(\omega)}{A_{\rm s}(\bfx)}(\I\omega)^{-\!\frac{n+1}{2}} [\I\omega+c(\bfx)\bfns(\bfx)\cdot\partial_\bfx]\widehat{\waux}(\bfx,\omega) \\
    K\waux(\bfx)                  &= \frac{1}{2\pi}\int\e^{\I\omega T_{\rm s}(\bfx)}\widehat{\waux}(\bfx,\omega)\, d\omega.
\end{split}
\end{equation}
Applied to the reverse time continued wave field $u_{\rm r}(\bfx,\omega)$, equation (\ref{equ_ISV_tildeR}) becomes
\begin{equation}    \label{equ_Gomega0}
   \widetilde{R} r(\bfx) = \frac{1}{2\pi}\int\e^{\I\omega T_{\rm s}(\bfx)}
   \frac{\tilde\sigma(\omega)}{A_{\rm s}(\bfx)}(\I\omega)^{-\!\frac{n+1}{2}} [\I\omega+c(\bfx)\bfns(\bfx)\cdot\partial_\bfx]
   \widehat{u}_{\rm r}(\bfx,\omega) \, d\omega.
\end{equation}

The next step is to eliminate $T_{\rm s}(\bfx)$, $A_{\rm s}(\bfx)$ and $\bfns(\bfx)$ by expressing them in terms of the source field explicitly. The principal term of the geometrical optics approximation of the source (\ref{equ_SCA_hot_source}) is
\begin{equation*}
   \widehat{g}(\bfx,\omega) = A_{\rm s}(\bfx) \sigma(\omega) (\I\omega)^{\frac{n-3}{2}} \, \e^{-\I\omega T_{\rm s}(\bfx)}.
\end{equation*}
Function $\sigma$, introduced in (\ref{equ_SCA_hot_source}), is smooth and has value 1 except for a small neighborhood of the origin where it is 0. Later we will examine the effect of the subprincipal terms of the source and the division by its amplitude. One naively derives the following identities
\begin{equation}  \label{equ_naive_elimination}
\begin{split}
    \e^{\I\omega T_{\rm s}(\bfx)} \frac{1}{A_{\rm s}(\bfx)} (\I\omega)^{-\!\frac{n+1}{2}} &= \frac{\sigma(\omega)}{(\I\omega)^2\widehat{g}(\bfx,\omega)} \\
    c(\bfx)\bfns(\bfx) &= \frac{c(\bfx)^2 \partial_\bfx\widehat{g}(\bfx,\omega)}{-\I\omega\widehat{g}(\bfx,\omega)}
    = \frac{c(\bfx)^2 \partial_\bfx\overline{\widehat{g}(\bfx,\omega)}}{\I\omega\overline{\widehat{g}(\bfx,\omega)}},
\end{split}
\end{equation}
in which it is used that the second equation is real-valued. Substitution of involved factors occurring in the integral (\ref{equ_Gomega0}) yields
\begin{equation}    \label{equ_Gomega}
   \widetilde{R} r(\bfx) = \frac{1}{2\pi}\int\frac{\tilde\sigma(\omega)\sigma(\omega)}{\I\omega\widehat{g}(\bfx,\omega)} \Big[
   1+\frac{c(\bfx)^2 \partial_\bfx\overline{\widehat{g}(\bfx,\omega)}\cdot\partial_\bfx}{(\I\omega)^2\overline{\widehat{g}(\bfx,\omega)}}
   \Big] \widehat{u}_{\rm r}(\bfx,\omega) \, d\omega.
\end{equation}

We will finally argue that the division by the source amplitude is well-defined and that the subprincipal terms in the expansion for $\widehat{g}(\bfx,\omega)$ do not affect the expression for the principal symbol (\ref{eq:prin_symb_R}).
The source wave field is free of caustics by assumption. The transport equation yields that, on a compact domain in spacetime, there exists a lower bound $L>0$ for the principal amplitude, thus $|A_0(\bfx,\bfx_{\rm s},\omega)|\ge L$. Division by $A_0$ is therefore well-defined, and from its homogeneity and the inequality (\ref{equ_SCA_asymptotic_sum}) it can be deduced that there exists a constant $C>0$ such that $
   \left| \frac{A(\bfx,\bfx_{\rm s},\omega)}{A_0(\bfx,\bfx_{\rm s},\omega)} - 1\right| \le \frac{C}{1+|\omega|}$.
For $|\omega|$ sufficiently large, division by $A$ is therefore well-defined. We choose $1\!-\Omega$ wide enough such that all \mbox{$\omega\in\supp(\Omega)$} are high and satisfy \mbox{$\tilde\sigma(\omega)\sigma(\omega)=1$}. The difference between $\tfrac{1}{A_0}$ and $\tfrac{1}{A}$ is of lower order in $\omega$. 
By construction it holds that \mbox{$A_0(\bfx,\bfx_{\rm s},\omega) = A_{\rm s}(\bfx)\,(\I\omega)^{\frac{n-3}{2}}$} on $\supp(\Omega)$. Taking (\ref{equ_Gomega}) we replace $\tilde\sigma\sigma$ with $\Omega$ to define the imaging condition (\ref{equ_Gomega_Omega}).
\end{proof}

\section{Numerical examples} \label{sec_numerics}

In this section, we give numerical examples to support our theorems.
The general setup of the examples was as follows. First a model was
chosen, consisting of a background medium $c$, a medium perturbation
(contrast) $\delta c = c r$, a domain of interest and a computational
domain. The latter was larger than the domain of interest and included
absorbing boundaries. Data were generated by solving the inhomogeneous
wave equation with velocity $c + \delta c$, and a Ricker wavelet
source signature at position $\bfx_{\rm s}=(0,0)$, using an order
(2,4) finite difference scheme \cite{Cohen2002}. The direct wave was
eliminated. The operator (\ref{eq:one_way_factor}) could be applied in
the Fourier domain since in the examples $c$ was constant at the
surface. The backpropagated field was then computed using the finite
difference method, and the same for the source field. Finally the
imaging condition (\ref{equ_Gomega_Omega}) was applied to obtain an
approximate reconstruction of $\delta c$.

As we mentioned, only a partial reconstruction of $\delta c$ is
possible in realistic situations. Relation (\ref{equ_SCA_coivar}) and
the wave propagation restrict the directions of $\bfzeta$ where
inversion is possible. The frequency range present in the data also
restricts the length of $\bfzeta$, according to (\ref{equ_SCA_coivar})
and using that $| \bfxi | = c^{-1} | \omega |$.  To be able to compare
the original and reconstructed reflectivity we used bandlimited
functions for $\delta c$, which where obtained by multiplying a plane
wave with a window function. Such functions are localized in position,
by the support of the window, and in wave vector by the plane wave.

Our first example concerns a gradient type medium with $c(x_1,x_2) =
2.0 + 0.001 x_2$ with $c$ in km/s and $x_2$ in meters.  Our model
region was the square with $x_1$ and $x_2$ between $0$ and $2000$
meters. The purpose was to show a successful reconstruction of
velocity perturbations at different positions and with different
orientations in the model. We therefore chose for $\delta c$ a linear
combination of three wave packets at different locations, with central
wave vector well within in the inversion aperture. We included one
with large dip, as one of the interesting abilities of RTM is imaging
of large dips. The results of the above procedure are shown in
Figures~\ref{fig:ex3dv,ex3migim} and
~\ref{fig:ex3tracesx400,ex3tracesx1400,ex3tracesz600}. The
reconstruction of the phase is excellent. However, the reconstructed
amplitude is around 8-10 \% smaller than the original amplitude.
Possible explanations for this are inaccuracies related to the
linearization and to a limited aperture.

\def\widthfig{6.5cm}
\begin{figure}[ht]
   \begin{center}
      \includegraphics[width=\widthfig]{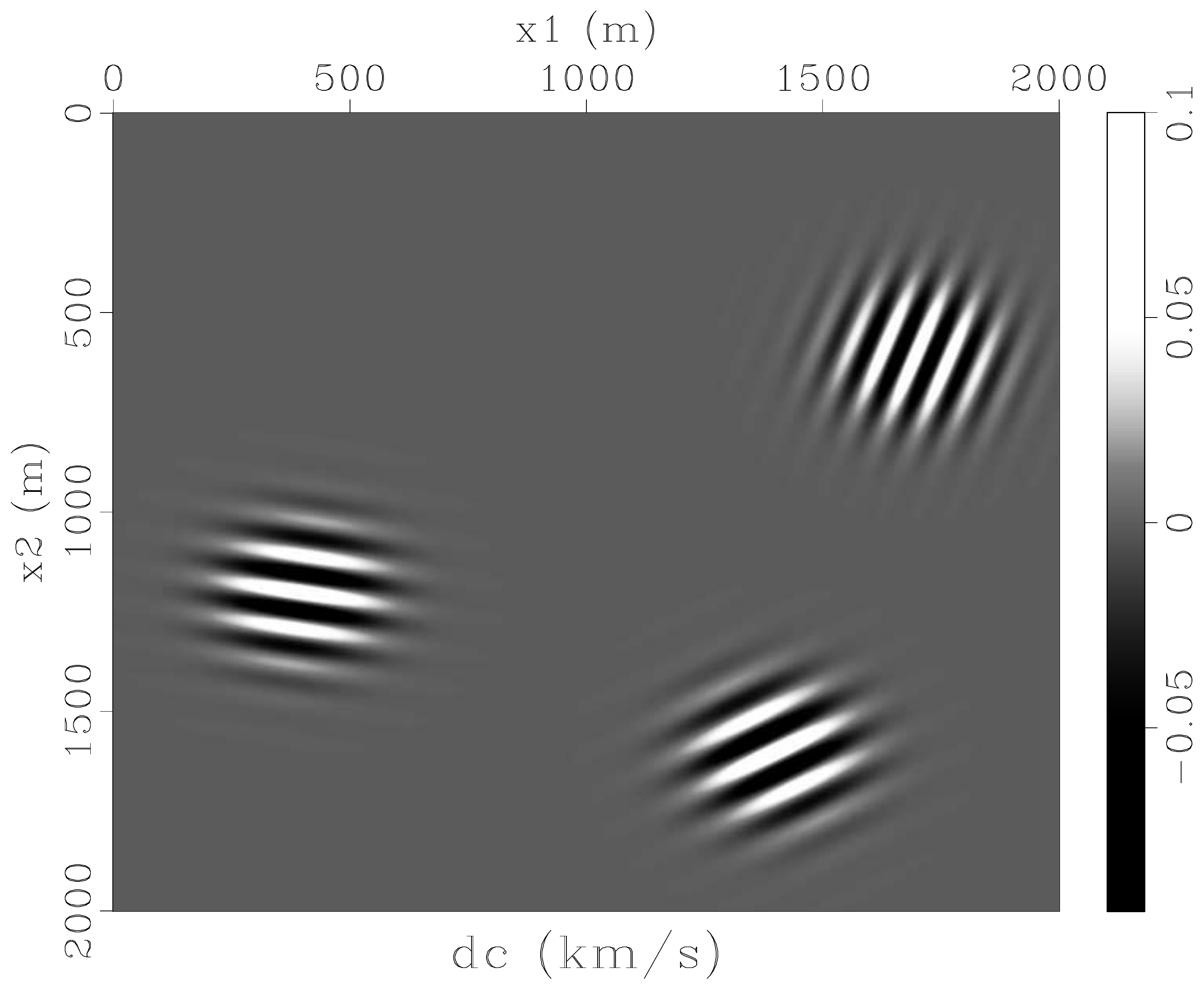}
      \hspace{1.5cm}
      \includegraphics[width=\widthfig]{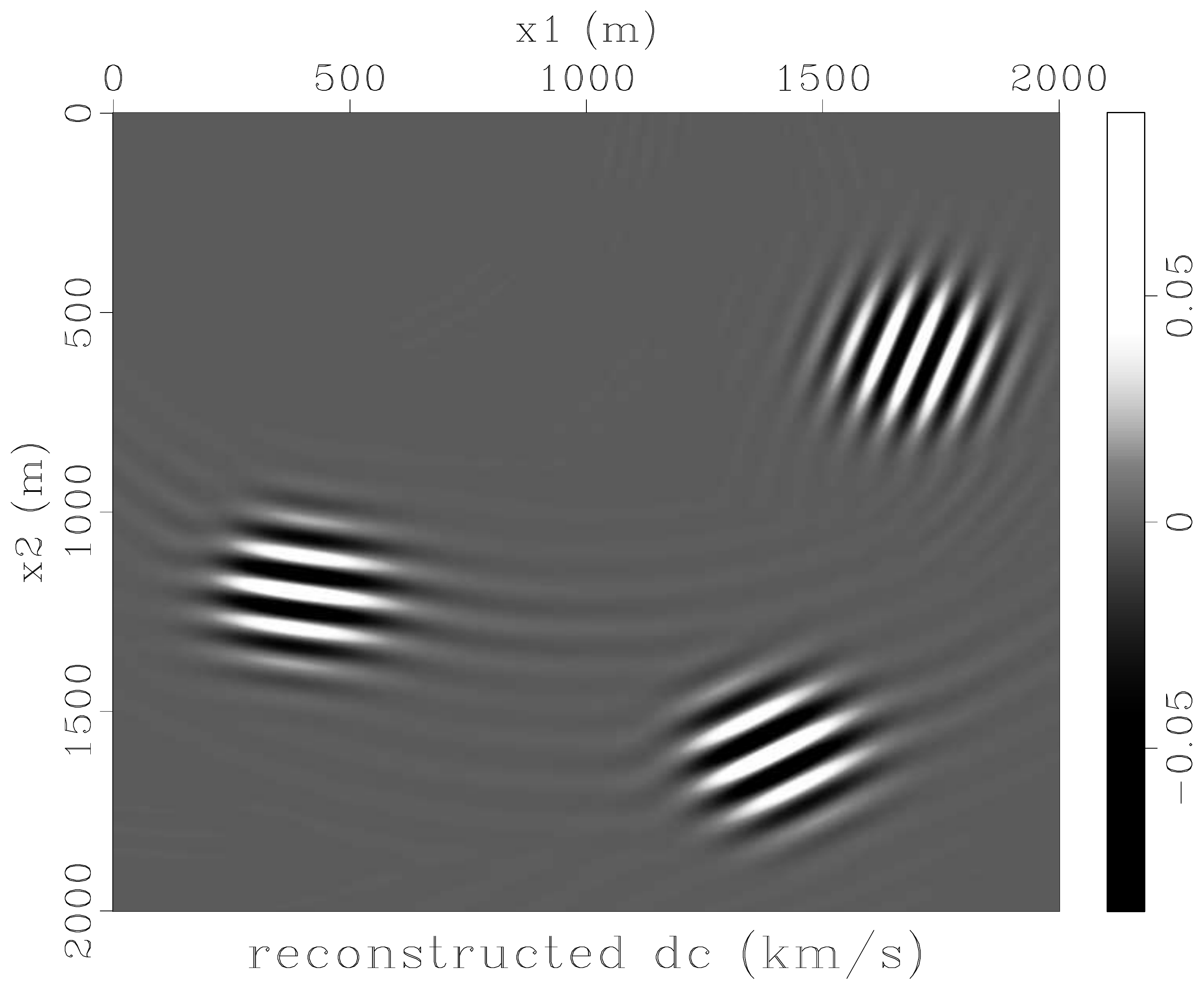}
      \\ \vspace{0.2cm} {\footnotesize (a) \hspace{7.5cm} (b)}
   \end{center}
   \caption{Example 1: Velocity perturbation and reconstructed velocity 
   perturbation. 
   The background medium is a gradient $c = 2.0 + 0.001 x_2$,
   with $x_2$ in meters and $c$ in km/s.}
   \label{fig:ex3dv,ex3migim}
\end{figure}
\begin{figure}[ht]
   \begin{center}
      \includegraphics[width=4.4cm]{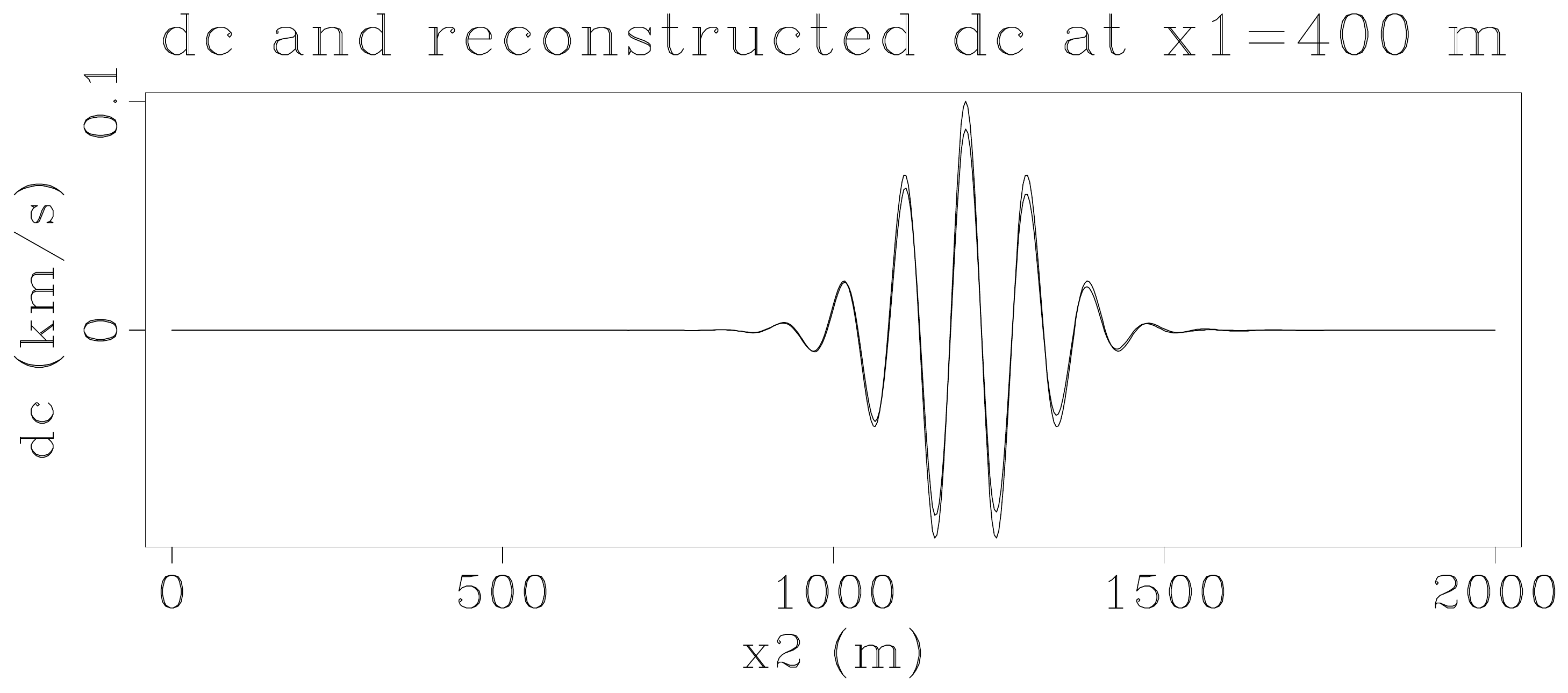}
      \hspace{0.6cm}
      \includegraphics[width=4.4cm]{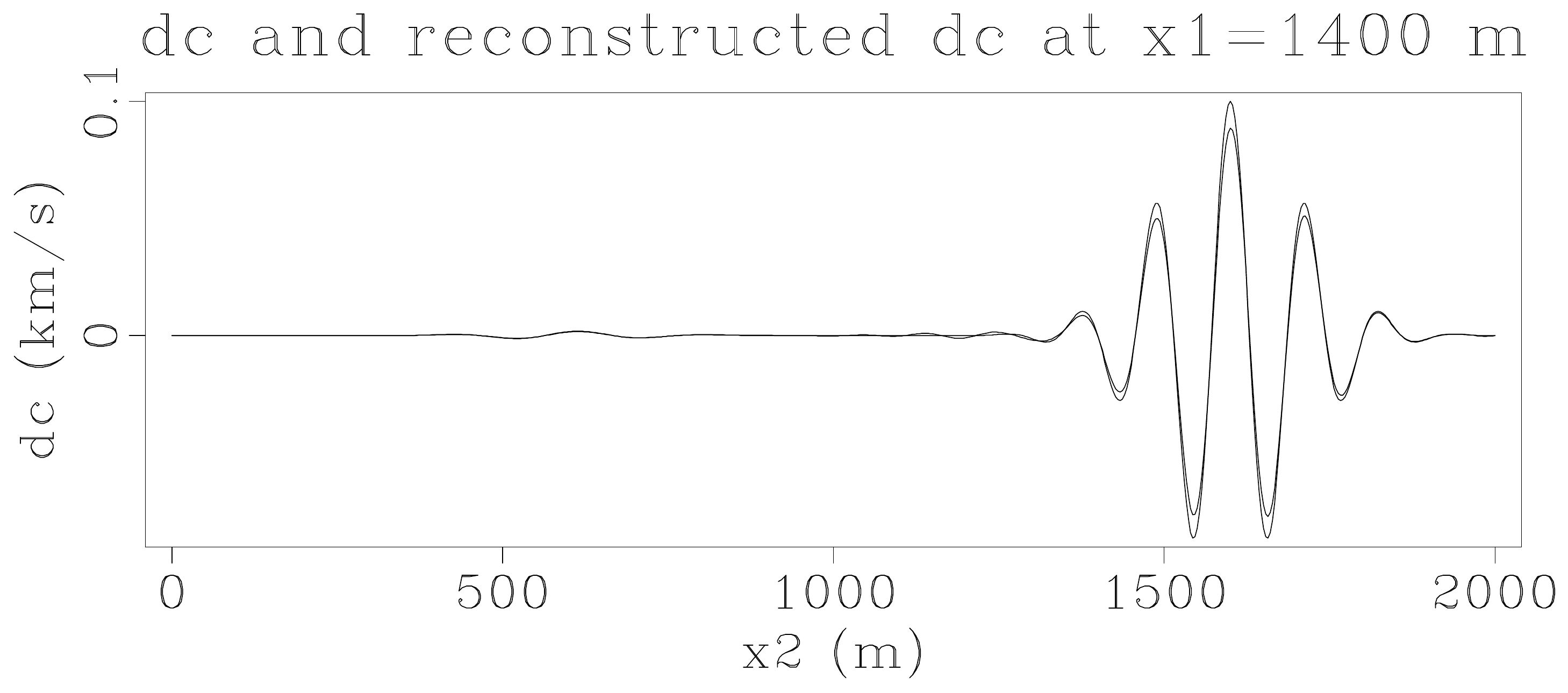}
      \hspace{0.6cm}
      \includegraphics[width=4.4cm]{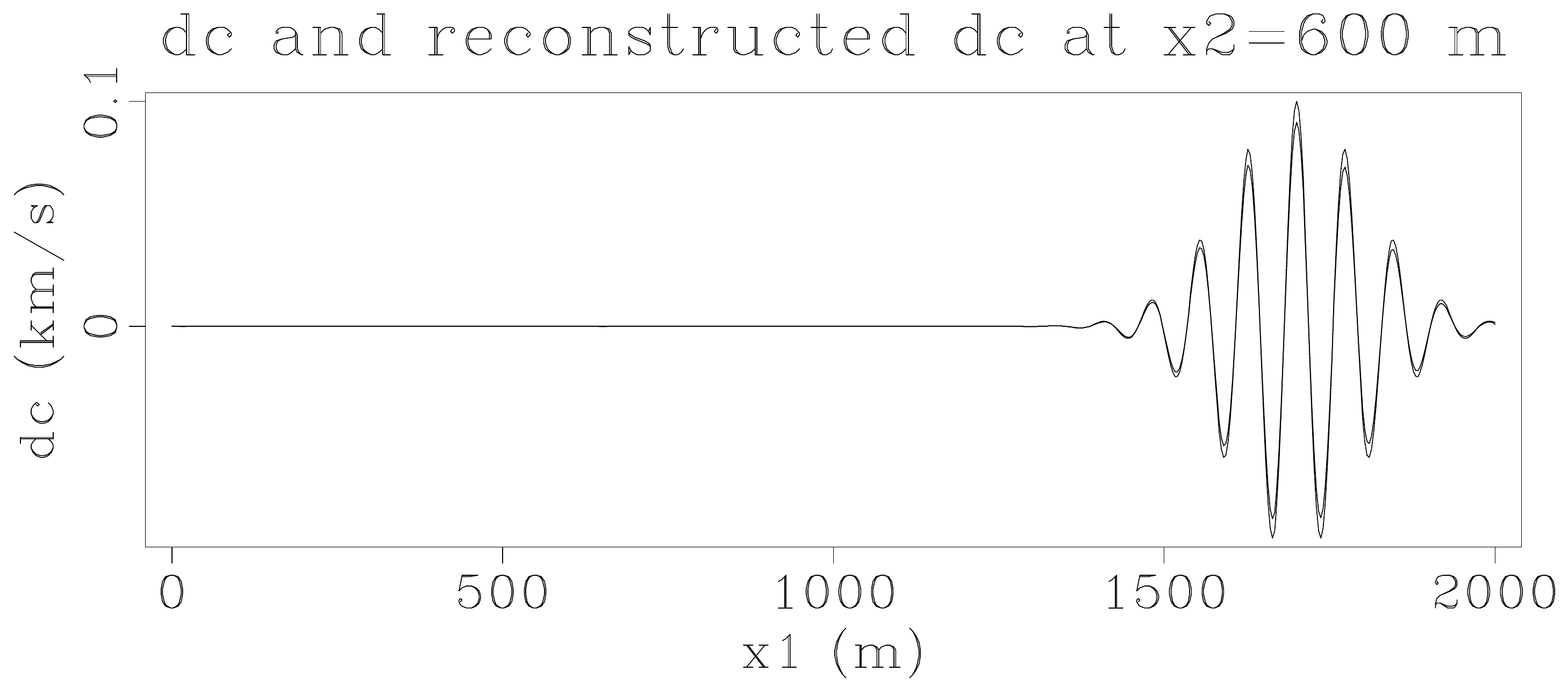}
   \end{center}
   \caption{Example 1: Comparison of some traces from 
   Figure~\ref{fig:ex3dv,ex3migim} at
   $x_1=400$ m, $x_1=1400$ m and $x_2=600$ m.}
   \label{fig:ex3tracesx400,ex3tracesx1400,ex3tracesz600}
\end{figure}

Our second example concerns a bandlimited continuous reflector. For a
continuous reflector one might expect less loss in amplitude when
compared to the localized velocity perturbations. One of the strengths
of RTM and wave equation migration in general is that multipathing is
easily incorporated, where in our case of single source RTM,
multipathing is only allowed between the reflector and the receiver
point. To see this in an example we included in our background model a
low velocity lens at $(800,1200)$ m.  The background medium including
some rays, as well as some data are plotted in
Figure~\ref{fig:ex6velrays,ex6simsingle}. The velocity perturbation
was located at $x_2=1600$~m.  The results of this example are given in
Figure~\ref{fig:ex6dv,ex6migim}. The reconstruction of the phase is
again excellent. The amplitude varies somewhat depending on location,
being about $0$-$10$ \% too low. The smooth tapering which was applied
has diminished smiles and amplitude variations, but not fully
eliminated them. The multipathing leads to singularities in the
inverse of the source field $\hat{u}_{\rm inc}^{-1}$, around
$(x_1,x_2) = (1900,1000)$ m, which leads to the two artifacts that can
be seen there.

\begin{figure}[ht]
   \begin{center}
      \includegraphics[width=\widthfig]{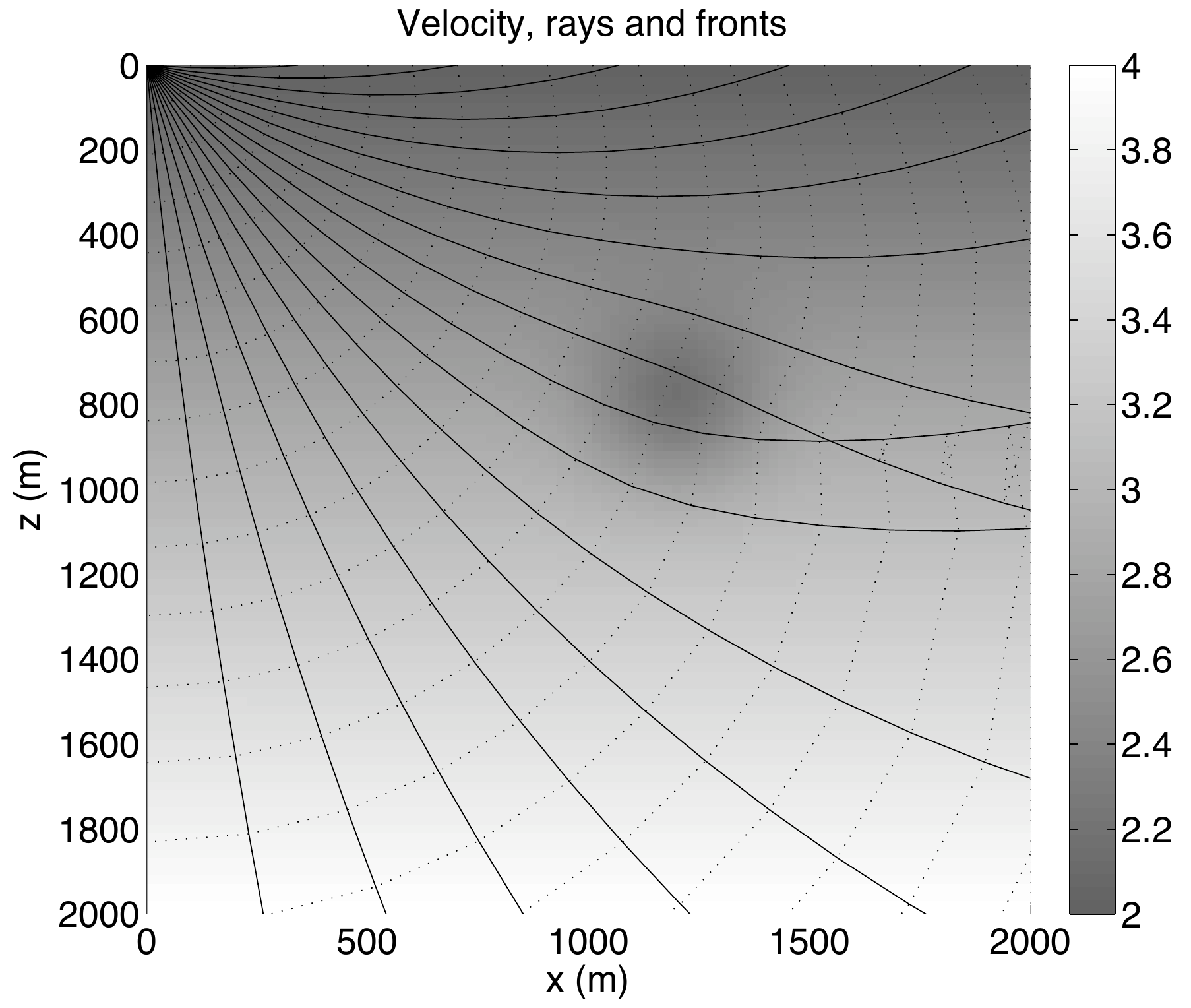}
      \hspace{1.5cm}
      \includegraphics[width=\widthfig]{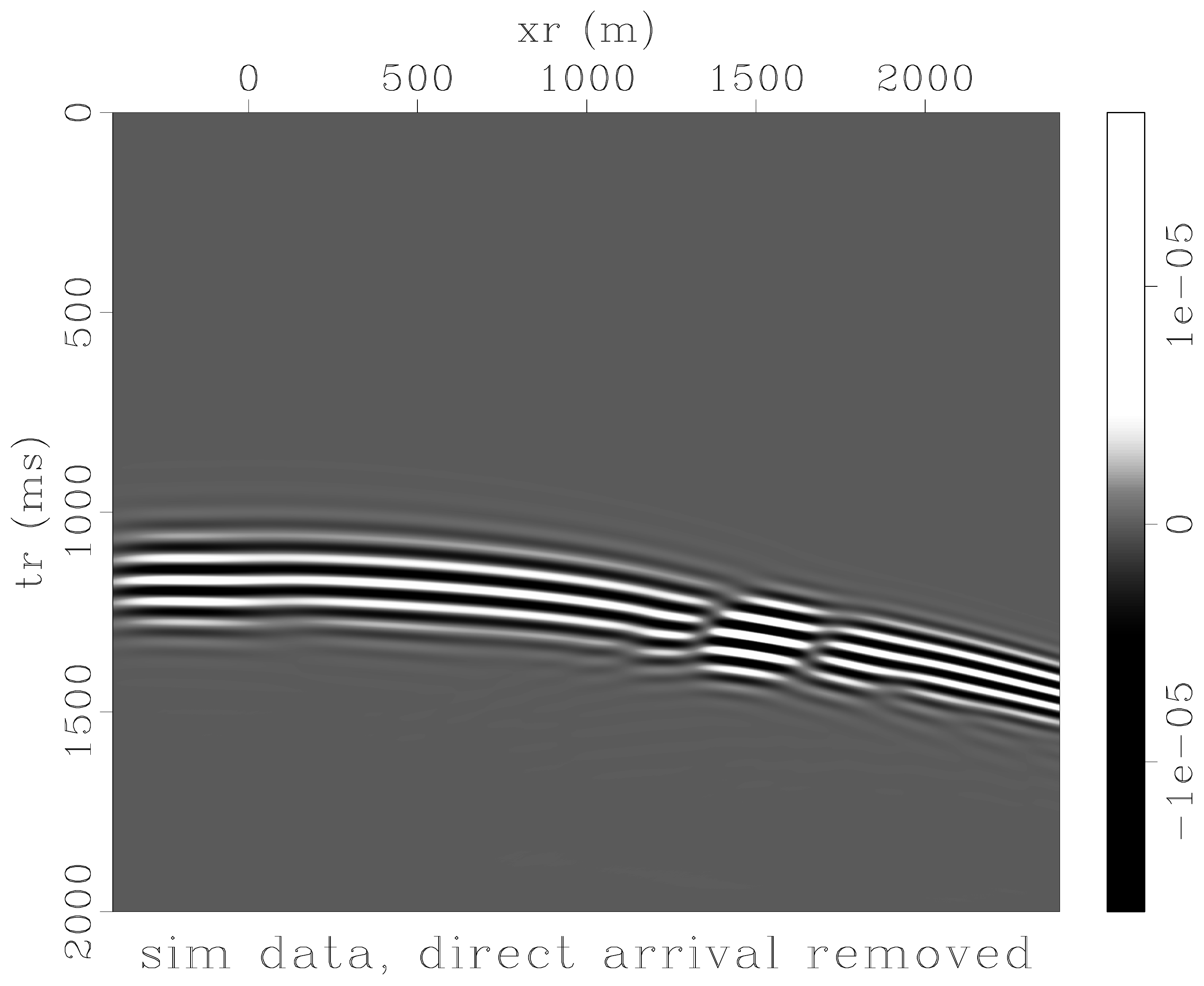}
      \\ \vspace{0.1cm} {\footnotesize (a) \hspace{7.5cm} (b)}
   \end{center}
   \caption{Example 2: (a) A velocity model with some rays; 
   (b) Simulated data, with direct arrival removed.}
   \label{fig:ex6velrays,ex6simsingle}
\end{figure}
\begin{figure}[ht]
   \begin{center}
      \includegraphics[width=\widthfig]{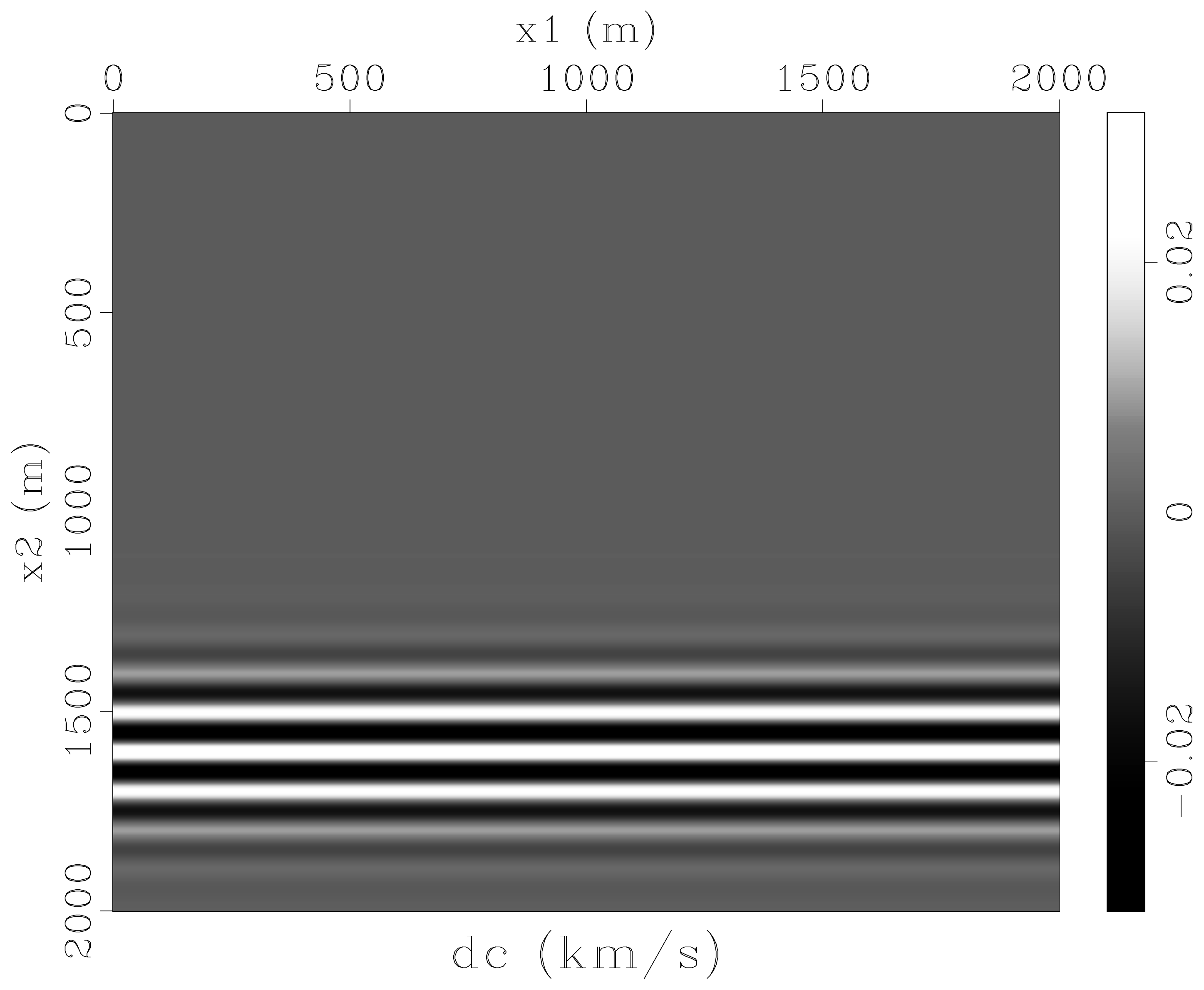}
      \hspace{1.5cm}
      \includegraphics[width=\widthfig]{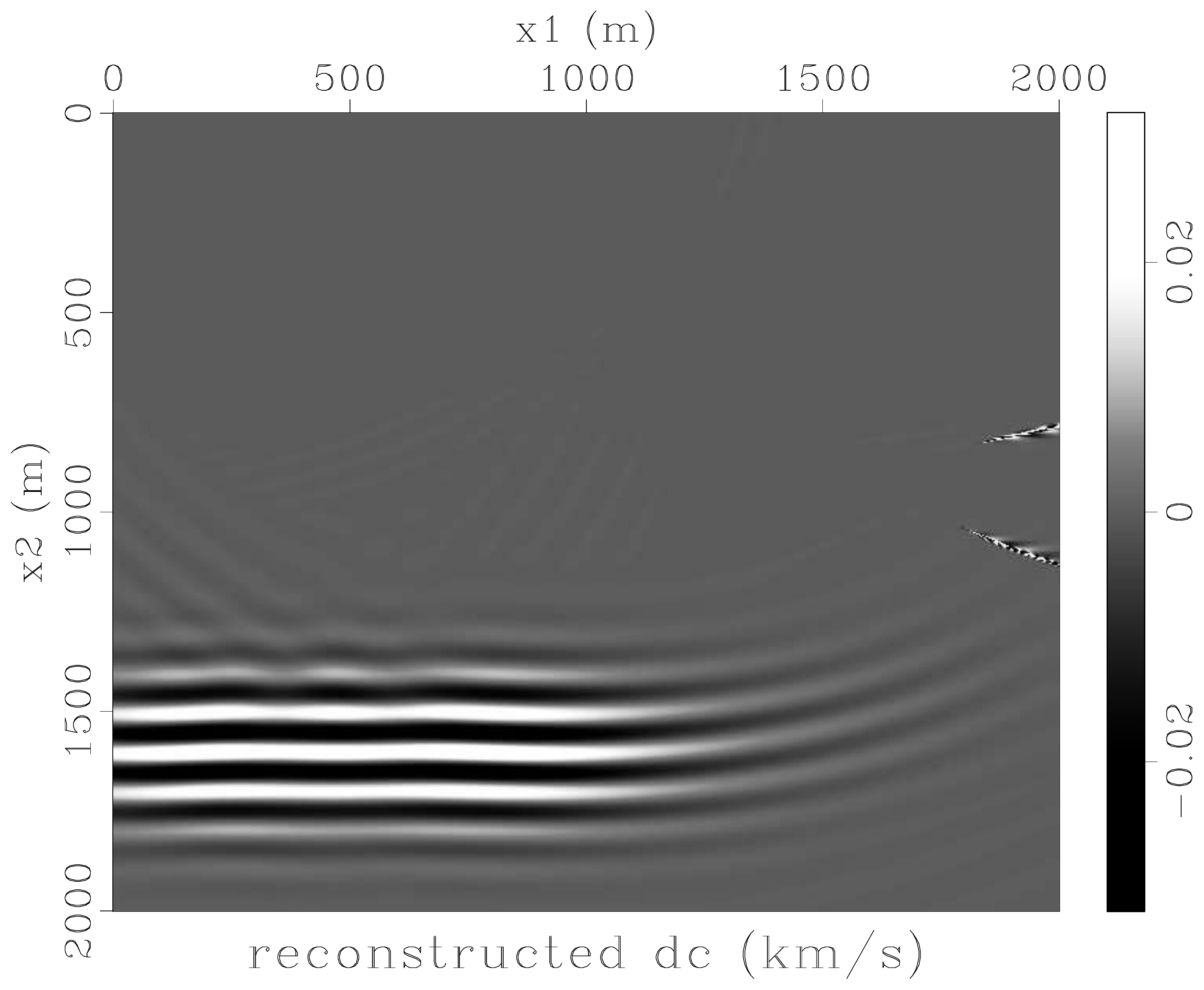}
      \\ \vspace{0.1cm} {\footnotesize (a) \hspace{7.5cm} (b)}
   \end{center}
   \caption{Example 2: (a) Velocity perturbation; 
   (b) Partial reconstruction of the velocity perturbation.}
   \label{fig:ex6dv,ex6migim}
\end{figure}

\section{Discussion}
\label{subsec_conc_disc}

We presented a comprehensive analysis of RTM-based imaging, and
introduced an imaging condition condition involving only local (data
point and image point) operators which yields a parametrix for the
single scattering problem for a given point source.

We make the following observations concerning our inverse scattering
procedure: (i) The symbol of the normal operator associated with a
single point source contains a singularity which has been observed in
the form of ``low-frequency'' artifacts \cite{yoonMS:2004,
  mulderP:2004, fletcherFKA:2005, Xie2006, guittonKB:2007}. 
Our imaging condition yields a parametrix and naturally avoids this singularity.
(ii) The square-root operator (\ref{eq:one_way_factor}), a factor of $F_M$ introduced in
section~\ref{sec_up_down_PsDO}, can be removed with dual sensor
(streamer) data, that is, if the surface-normal derivative of the wave
field is measured. We note that $F_M$ is available only
microlocally. (iii) Division by the source field, in frequency, can
lead to poor results when its amplitude is small. There are two main
reasons why this can occur. First, a realistic source signature can
yield very small values for particular frequencies in its amplitude
spectrum. Moving averaging in frequency typically resolves this
situation \cite{Guitton2006, Chattopadhyay2008, Costa2009}. Secondly,
the illumination due to propagation in a velocity model of high
complexity may result in small values; spatial averaging over small
neighborhoods of the image points may be benificial. (The
cross-correlation imaging has been adapted by normalization with the
source wave field energy at the imaging points as a proxy to inverse
scattering \cite{Claerbout1971, Biondi2006}.)

The acquisition aperture, and associated illumination, is intimately
connected to the resolution operator $R$. This operator is
pseudodifferential and the support of its symbol expresses which parts
of the contrast or reflectivity can be recovered from the available
data. Partial reconstruction is optimally formulated in terms of
curvelets or wave packets. A detailed procedure, making use of the
fact that the single scattering or imaging operator is associated with
a canonical graph, can be found in \cite{dHSUvdH}; see also
\cite{HerrmannMoghaddamStolk2008}.

We have addressed the single-source acquisition geometry, which arises
naturally in RTM. One can anticipate an immediate extension of our
reconstruction to multi-source data, but a major challenge arises
because the single source reconstructions are only partial. Because
each of the single source images result in reconstructions at
different sets of points and orientations, in general, which are not
identified within the RTM algorithm, averaging must be
excluded. However, techniques from microlocal analysis can be invoked
to properly exploit the discrete multi-source acquisition
geometry. (We note that in the case of open sets of sources the
generation of source caustics will be allowed.)

\subsection*{Acknowledgement}

CCS and TJPMOtR were supported by the Netherlands Organisation for
Scientific Research through VIDI grant 639.032.509.  MVdH was
supported in part by the members of the Geo-Mathematical Imaging Group
at Purdue University.

\bibliography{references}
\bibliographystyle{plain}

\end{document}